\documentclass{amsart}
\usepackage{amssymb,amsmath}
\usepackage{stmaryrd}
\usepackage{yhmath}
\usepackage{latexsym}
\usepackage{amsthm}
\usepackage{amscd}
\usepackage{mathrsfs}
\usepackage[all]{xy}
\usepackage{mathtools}
\usepackage{arydshln}
\usepackage{bbm}
\usepackage{bm}
\usepackage{url}
\usepackage{color}
\usepackage{colonequals}

\newcommand*{\printtime}{\the\year/\the\month/\the\day}

\newcommand{\lan}{\langle}
\newcommand{\ran}{\rangle}
\newcommand{\ol}{\overline}

\newcommand{\temp}{\mathrm{temp}}

\newcommand{\mcO}{\mathcal{O}}

\newcommand{\Z}{\mathbb{Z}}
\newcommand{\F}{\mathbb{F}}

\newcommand{\Q}{\mathbb{Q}}
\newcommand{\C}{\mathbb{C}}

\newcommand{\x}{\mathbf{x}}

\newcommand{\G}{\mathbf{G}}

\newcommand{\mfp}{\mathfrak{p}}

\DeclareMathOperator{\Tr}{Tr}

\DeclareMathOperator{\val}{val}

\DeclareMathOperator{\GL}{GL}

\DeclareMathOperator{\SO}{SO}

\DeclareMathOperator{\SL}{SL}
\DeclareMathOperator{\Sp}{Sp}

\DeclareMathOperator{\cInd}{c-Ind}
\DeclareMathOperator{\Ind}{Ind}

\DeclareMathOperator{\Hom}{Hom}

\DeclareMathOperator{\Ad}{Ad}

\DeclareMathOperator{\Kl}{Kl}
\DeclareMathOperator{\Cent}{Cent}

\DeclareMathOperator{\diag}{diag}

\DeclareMathOperator{\Swan}{Swan}
\DeclareMathOperator{\Artin}{Artin}
\DeclareMathOperator{\Frob}{Frob}

\DeclareMathOperator{\Sym}{Sym}
\DeclareMathOperator{\LLC}{LLC}

\theoremstyle{plain}
\newtheorem{thm}{Theorem}[section]
\newtheorem*{thm*}{Theorem}
\newtheorem{prop}[thm]{Proposition}
\newtheorem{lem}[thm]{Lemma}
\newtheorem{cor}[thm]{Corollary}
\newtheorem{conj}[thm]{Conjecture}

\theoremstyle{definition}
\newtheorem{defn}[thm]{Definition}

\theoremstyle{remark}
\newtheorem{rem}[thm]{Remark}
\newtheorem{example}[thm]{Example}
\newtheorem*{claim*}{Claim}

\SelectTips{cm}{11}

\title{Simple supercuspidal $L$-packets of symplectic groups over dyadic fields}

\author{Guy Henniart}
\address{Universit\'e Paris-Saclay, CNRS, Laboratoire de Math\'ematiques d’Orsay, 91405, Orsay, France.}
\email{Guy.Henniart@math.u-psud.fr}

\author{Masao Oi}
\address{Department of Mathematics (Hakubi center), Kyoto University, Kitashirakawa, Oiwake-cho, Sakyo-ku, Kyoto 606-8502, Japan.}
\email{masaooi@math.kyoto-u.ac.jp}

\begin{document}

\begin{abstract}
We consider the symplectic group $\Sp_{2n}$ defined over a $p$-adic field $F$, where $p=2$.
We prove that every simple supercuspidal representation (in the sense of Gross--Reeder) of $\Sp_{2n}(F)$ corresponds to an irreducible $L$-parameter under the local Langlands correspondence for $\Sp_{2n}$ established by Arthur.
\end{abstract}

\maketitle

\section{Introduction}
Let $F$ be a $p$-adic field, where $p$ is a prime number.
Let $\G$ be a split connected reductive group over $F$.
\textit{The local Langlands correspondence for $\G$}, which is still conjectural in general, asserts that there exists a natural surjective map
\[
\LLC_{\G}\colon\Pi(\G)\rightarrow\Phi(\G)
\]
with finite fibers, where
\begin{itemize}
\item
$\Pi(\G)$ denotes the set of equivalence classes of irreducible admissible representations of $\G(F)$, and
\item
$\Phi(\G)$ denotes the set of $\hat{\G}$-conjugacy classes of $L$-parameters of $\G$.
\end{itemize}
In other words, it is expected that the set $\Pi(\G)$ can be partitioned into the disjoint union of finite sets $\Pi^{\G}_{\phi}:=\LLC_{\G}^{-1}(\phi)$ (called \textit{$L$-packets}) labelled by $L$-parameters $\phi\in\Phi(\G)$:
\[
\Pi(\G)
=
\bigsqcup_{\phi\in\Phi(\G)}\Pi_{\phi}^{\G}.
\]
Here, recall that an $L$-parameter of $\G$ is a homomorphism $W_{F}\times\SL_{2}(\C)\rightarrow\hat{\G}$ with certain conditions (see Section \ref{subsec:LLC}), where $W_{F}$ is the Weil group of $F$ and $\hat{\G}$ is the Langlands dual group of $\G$ over $\C$.

For several specific groups, the local Langlands correspondence has been established completely.
Especially, when $\G$ is $\GL_{N}$, the correspondence was constructed by Harris--Taylor \cite{MR1876802} and the first author \cite{MR1738446}.
Also, when $\G$ is a symplectic or special orthogonal group, the correspondence was constructed by Arthur (\cite{MR3135650}).
However, since their methods are based on geometric or global tools, it is not obvious how the map $\LLC_{\G}$ can be described explicitly.
Hence it is quite natural to seek an explicit description of the map $\LLC_{\G}$ for the above-mentioned groups.
Indeed, in the case where $\G=\GL_{N}$, a lot of studies have been carried out by many people so far, as represented by the consecutive work of Bushnell--Henniart (\cite{MR2138141,MR2148193,MR2679700}).

In this paper, we consider this problem in the case where $\G=\Sp_{2n}$ and $p=2$.
Let us explain why this case is particularly of our interest.
When $\G=\Sp_{2n}$, the Langlands dual group of $\Sp_{2n}$ is given by $\SO_{2n+1}(\C)$.
Hence an $L$-parameter of $\Sp_{2n}$ is regarded as a $(2n+1)$-dimensional orthogonal representation of $W_{F}\times\SL_{2}(\C)$.
In fact, Arthur's theorem (\cite{MR3135650}) also asserts that the $L$-packet $\Pi_{\phi}^{\G}$ for each $L$-parameter $\phi$ is equipped with a bijection to the set of irreducible characters of a finite group $\mathcal{S}_{\phi}$ defined by
\[
\mathcal{S}_{\phi}:=\pi_{0}\bigl(\Cent_{\SO_{2n+1}(\C)}(\mathrm{Im}(\phi))\bigr).
\]
This implies that, for example, if $\phi$ is a direct sum of $m$ inequivalent irreducible orthogonal representations of $W_{F}\times\SL_{2}(\C)$, then the $L$-packet $\Pi_{\phi}^{\Sp_{2n}}$ consists of $2^{m-1}$ elements.
If furthermore $\phi$ is trivial on $\SL_{2}(\C)$ in this situation, then all of members of $\Pi_{\phi}^{\Sp_{2n}}$ are supercuspidal by a result of Xu \cite{MR3713922} (see \cite[Section 4]{Oi:2018}).

However, when $p\neq2$, it is known that there is no irreducible orthogonal representation of $W_{F}$ whose dimension is odd and greater than $1$ (see, e.g., \cite[84page, Proposition 4]{MR1670568}).
Thus we can conclude that there is no singleton $L$-packet of $\Sp_{2n}$ consisting of a supercuspidal representation when $p\neq2$.

On the other hand, when $p=2$, there exist plenty of irreducible orthogonal representations of $W_{F}$ whose dimension is odd and greater than $1$; indeed, Bushnell--Henniart gave a complete classification of such representations (\cite{MR2824846}).
Therefore there should be singleton supercuspidal $L$-packets of $\Sp_{2n}$.
This phenomenon can occur only when $p=2$.

In \cite{Henniart:2022}, the first author proved that when $F=\Q_{2}$, the supercuspidal representations which are \textit{simple} in the sense of Gross--Reeder (\cite{MR2730575}, see Section \ref{sec:ssc}), give such an $L$-packet of $\Sp_{2n}(\Q_{2})$ .
The aim of this paper is to extend it to any dyadic field (i.e., a finite extension of $\Q_{2}$).
Our main result is the following:

\begin{thm}[Theorem \ref{thm:main} (1) and Corollary \ref{cor:main}]\label{thm:intro-main}
Let $F$ be a dyadic field.
The $L$-parameter of any simple supercuspidal representation of $\Sp_{2n}(F)$ is irreducible as a $(2n+1)$-dimensional representation of $W_{F}$.
In particular, any $L$-packet of $\Sp_{2n}$ containing a simple supercuspidal representation of $\Sp_{2n}(F)$ is a singleton.
\end{thm}

We remark that the $L$-parameter of a simple supercuspidal representation of $\Sp_{2n}(F)$ can be explicited by appealing to a result of Bushnell--Henniart (see Remark \ref{rem:explicit}).

We explain the outline of our proof.
Our method is based on the \textit{twisted endoscopic character relations} between $\Sp_{2n}$ and $\GL_{2n+1}$ and is similar to that of \cite{MR3904769}, in which the second author obtained a result of the same type in the case where $\G=\SO_{2n+1}$ and $p$ is odd.

We first take a $\theta$-stable simple supercuspidal representation $\pi$ of $\GL_{2n+1}(F)$ with trivial central character, where $\theta$ is a suitable involution of $\GL_{2n+1}$ (such a representation exists only when $p=2$).
In fact, the equivalence classes of such simple supercuspidal representations can be explicitly parametrized by the finite set $k^{\times}$, where $k$ is the residue field of $F$ (see Section \ref{subsec:ssc-GL}).
Hence let us suppose that $\pi$ is the simple supercuspidal representation $\pi^{\GL_{2n+1}}_{a}$ labelled by $a\in k^{\times}$.
By the local Langlands correspondence for $\GL_{2n+1}$, we get the $L$-parameter $\phi_{a}$ corresponding to $\pi^{\GL_{2n+1}}_{a}$, which is irreducible orthogonal as a $(2n+1)$-dimensional representation, is trivial on $\SL_{2}(\C)$, and has trivial determinant.
Then, by regarding $\phi_{a}$ as an $L$-parameter of $\Sp_{2n}$, we get a singleton $L$-packet $\Pi_{\phi}^{\Sp_{2n}}$ (in the sense of Arthur) consisting of a supercuspidal representation $\pi^{\Sp_{2n}}$ of $\Sp_{2n}(F)$ as explained above.

In this situation, Arthur's theory guarantees that the endoscopic character relation between $\pi^{\GL_{2n+1}}_{a}$ and $\Pi_{\phi}^{\Sp_{2n}}$ holds (see Section \ref{subsec:LLC}).
We first compute the $\theta$-twisted character of $\pi^{\GL_{2n+1}}_{a}$ at some specific regular semisimple elements which we call \textit{$\theta$-affine generic elements} of $\GL_{2n+1}(F)$.
Then, by applying the endoscopic character relation between $\pi^{\GL_{2n+1}}_{a}$ and $\Pi_{\phi}^{\Sp_{2n}}$ to $\theta$-affine generic elements, we get a description of the character of $\pi^{\Sp_{2n}}$ at affine generic elements of $\Sp_{2n}(F)$.
From this, we see that $\pi^{\Sp_{2n}}$ is either depth-zero supercuspidal or simple supercuspidal.

Our next task is to exclude the possibility that $\pi^{\Sp_{2n}}$ is depth-zero supercuspidal.
For this, we utilize the formal degree conjecture of Hiraga--Ichino--Ikeda (\cite{MR2350057}), which relates the formal degree of $\pi^{\Sp_{2n}}$ to the special value of the adjoint $\gamma$-factor of the $L$-parameter $\phi_{a}$.
(Recently, it was announced by Beuzart-Plessis that the formal degree conjecture was solved for $\Sp_{2n}$; see Section \ref{subsec:FDC}).
The point is that the latter invariant can be expressed via the Swan conductor of the exterior square of the $(2n+1)$-dimensional representation $\phi_{a}$ of $W_{F}$.
Since it can be computed by using a formula of Bushnell--Henniart--Kutzko (\cite{MR1606410}), the formal degree conjecture enables us to access the formal degree of $\pi^{\Sp_{2n}}$.
In fact, this is enough to conclude that $\pi^{\Sp_{2n}}$ is not depth-zero, hence is simple supercuspidal.

Finally, we note that the equivalence classes of simple supercuspidal representations of $\Sp_{2n}(F)$ are also parametrized by $k^{\times}$ (see Section \ref{subsec:ssc-Sp}).
Once we know that $\pi^{\Sp_{2n}}$ is simple supercuspidal, it is not hard to see that $\pi^{\Sp_{2n}}$ is the simple supercuspidal representation $\pi^{\Sp_{2n}}_{a}$ labelled by $a\in k^{\times}$.
In particular, we see that this descent construction exhausts all simple supercuspidal representations of $\Sp_{2n}(F)$.
This completes the proof.

Let us finish this introduction by giving several supplementary remarks:
\begin{rem}
\begin{enumerate}
\item
As mentioned above, the first author proved Theorem \ref{thm:intro-main} in the case where $F=\Q_{2}$ in his earlier paper \cite{Henniart:2022}.
However, the approach there is different from the one in this paper.
One of the keys in \cite{Henniart:2022} is a result of Adrian--Kaplan (\cite{MR4031300}), which needs that $F=\Q_{2}$.
\item
In \cite{Oi:2018}, the second author determines the structure of an $L$-packet containing a simple supercuspidal representation of $\Sp_{2n}(F)$ and its $L$-parameter in the case where $p\neq2$.
In this case, such an $L$-packet consists of two simple supercuspidal representations.
\item
The case where $n=1$, i.e., that of $\SL_{2}(F)$, is easier, and is a consequence of the local Langlands correspondence for $\SL_{2}(F)$ proved by Kutzko.
\item
It is worth noting that the \textit{depth} of any simple supercuspidal representation of $\Sp_{2n}(F)$ is given by $\frac{1}{2n}$ while the depth of its $L$-parameter is given by $\frac{1}{2n+1}$.
Thus, our result provides a counterexample to the depth preserving property of the local Langlands correspondence (cf.\ \cite{MR3618046}).
\end{enumerate}
\end{rem}

The organization of this paper is as follows.
In Section \ref{sec:ssc}, we explain a classification of simple supercuspidal representations of our concern.
In Section \ref{sec:char}, we compute the ($\theta$-twisted) characters of simple supercuspidal representations at ($\theta$-)affine generic elements.
In Section \ref{sec:descent}, we prove our main theorem by analyzing the twisted endoscopic character relation.

\medbreak
\noindent{\bfseries Acknowledgment.}\quad 
The second author was supported by JSPS KAKENHI Grant Number 20K14287.

\tableofcontents

\noindent
\textbf{Notations.}\quad
In this article, we let $p=2$.
Let $F$ be a $p$-adic field, $\mcO$ its ring of integers, $\mfp$ its maximal ideal, and $k$ its residue field $\mcO/\mfp$.
We write $q$ for the cardinality of $k$.
We often regard $k^{\times}$ as the subgroup of $F^{\times}$ consisting of elements of finite prime-to-$p$ order via the Teichm\"uller lift.
We fix a uniformizer $\varpi$ of $F$.
For any element $x\in\mcO$, we write $\overline{x}$ for its image in $k$.

We let $\psi\colon k\rightarrow\C^{\times}$ be the non-trivial additive character defined by $\psi=\psi_{\F_{2}}\circ\Tr_{k/\F_{2}}(\overline{x})$, where $\psi_{\F_{2}}$ is the unique nontrivial additive character of $\F_{2}$.
Note that $\psi$ is invariant under the Frobenius, i.e., $\psi(x^{2})=\psi(x)$ for any $x\in k$.

We let $I_{N}$ denote the identity matrix of size $N$ and $J_{N}$ denote the anti-diagonal matrix of size $N$ whose $(i, N+1-i)$-th entry is given by $(-1)^{i-1}$:
\[
J_{N} = \begin{pmatrix}
 &&&1\\
 &&-1&\\
 &\adots&&\\
 (-1)^{N-1}&&&
\end{pmatrix}.
\]

\section{Simple supercuspidal representations}\label{sec:ssc}

Let $\G$ be a split connected reductive group over $F$.
The \textit{simple supercuspidal representations} of $\G(F)$, which were introduced by \cite{MR2730575} and \cite{MR3164986}, are supercuspidal representations (with coefficients in $\C$) obtained by the compact induction of affine generic characters of Iwahori subgroups.
See \cite[Sections 2.1 and 2.2]{Oi:2018} for a general recipe and definition of simple supercuspidal representations.
In this section, we summarize a classification of
\begin{itemize}
\item
$\theta$-stable simple supercuspidal representations of $\GL_{2n+1}(F)$ with trivial central character, and
\item
simple supercuspidal representations of $\Sp_{2n}(F)$.
\end{itemize}
%Here, $\theta$ is the involution of $\GL_{2n+1}$ given by $\theta(g)=J_{2n+1}{}^{t}g^{-1}J_{2n+1}^{-1}$.
The classification is basically the same as the one given in \cite{MR3904769, Oi:2018}, but requires a minor modification because of the assumption that $p=2$.

\subsection{The case of $\GL_{2n+1}$}\label{subsec:ssc-GL}
Let us first work with $\GL_{N}$, where $N$ is not necessarily odd.
We also drop the assumption that $p=2$ temporarily.
Let $I_{\GL_{N}}$ be the standard Iwahori subgroup of $\GL_{N}$:
\[
I_{\GL_{N}}
=
\begin{pmatrix}
 \mcO^{\times}&&\mcO\\
 &\ddots&\\
 \mfp&&\mcO^{\times}
\end{pmatrix}.
\]
We let $I_{\GL_{N}}^{+}$ and $I_{\GL_{N}}^{++}$ denote the next two steps of the Moy--Prasad filtration subgroups of $I_{\GL_{N}}$:
\[
I_{\GL_{N}}^{+} = \begin{pmatrix}
 1+\mfp&&\mcO\\
 &\ddots&\\
 \mfp&&1+\mfp
\end{pmatrix}
\supset
I_{\GL_{N}}^{++} = \begin{pmatrix}
 1+\mfp&\mfp&&\mcO\\
 &\ddots&\ddots&\\
 &\mfp&\ddots&\mfp\\
\mfp^2&&&1+\mfp
\end{pmatrix}.
\]
Then we have 
\begin{align*}
I_{\GL_{N}}^{+}/I_{\GL_{N}}^{++} &\cong k^{\oplus N} \\
(x_{ij})_{ij} &\mapsto \bigl(\ol{x_{1,2}}, \ldots, \ol{x_{N-1, N}}, \ol{x_{N,1}\varpi^{-1}}\bigr)
\end{align*}
and the normalizer $N_{\GL_{N}(F)}(I_{\GL_{N}})$ of $I_{\GL_{N}}$ in $\GL_{N}(F)$ is given by
\[
N_{\GL_{N}(F)}(I_{\GL_{N}})
=
ZI_{\GL_{N}}\lan\varphi_{a}^{\GL_{N}}\rangle,
\]
for any $a\in k^{\times}$.
Here $Z$ denotes the center of $\GL_{N}(F)$ and, for $a \in k^{\times}$, we put 
\[
\varphi_{a}^{\GL_{N}} := 
\begin{pmatrix}
0 & I_{N-1} \\
\varpi a & 0 
\end{pmatrix} \in \GL_{N}(F)
\]
(note that $(\varphi_{a}^{\GL_{N}})^{N}$ equals the scalar matrix $\varpi a I_{N}$).

For $(a,\zeta)\in k^{\times}\times\mu_{N}$ (where $\mu_{N}$ denotes the set of $N$-th roots of unity in $\C^{\times}$), we define an affine generic character $\tilde{\chi}^{\GL_{N}}_{a,\zeta}\colon ZI_{\GL_{N}}^{+}\lan\varphi_{a^{-1}}^{\GL_{N}}\rangle\rightarrow\C^{\times}$ by
\begin{itemize}
\item
$\tilde{\chi}^{\GL_{N}}_{a,\zeta}(z):=1$ for $z\in Z$,
\item
$\tilde{\chi}^{\GL_{N}}_{a,\zeta}(\varphi_{a^{-1}}^{\GL_{N}}):=\zeta$, 
\item
$\tilde{\chi}^{\GL_{N}}_{a,\zeta}(x):=\psi(\ol{x_{1,2}}+\cdots+\ol{x_{N-1, N}}+a\ol{x_{N,1}\varpi^{-1}})$ for $x=(x_{ij})_{ij} \in I_{\GL_{N}}^{+}$.
\end{itemize}
Let $\pi_{a,\zeta}^{\GL_{N}}$ be the representation of $\GL_{N}(F)$ defined by
\[
\pi_{a,\zeta}^{\GL_{N}}:=\cInd^{\GL_{N}(F)}_{ZI_{\GL_{N}}^{+}\lan\varphi_{a^{-1}}^{\GL_{N}}\rangle} \tilde{\chi}^{\GL_{N}}_{a,\zeta}.
\]
Then the set
\[
\{
\pi^{\GL_{N}}_{a,\zeta}
\mid
(a,\zeta)\in k^{\times}\times\mu_{N}
\}
\]
represents the set of equivalence classes of simple supercuspidal representations of $\GL_{N}(F)$ with trivial central character (see \cite[Section 2.3]{MR3904769}).

Now let us suppose that $N=2n+1$.
We define an involution $\theta$ of $\GL_{2n+1}$ over $F$ by
\[
\theta(g):=J_{2n+1}{}^{t}g^{-1}J_{2n+1}^{-1}.
\]
The involution $\theta$ of $\GL_{2n+1}(F)$ preserves $I^{+}_{\GL_{2n+1}}$ and $I^{++}_{\GL_{2n+1}}$.
The action of $\theta$ induced on the quotient $I^{+}_{\GL_{2n+1}}/I^{++}_{\GL_{2n+1}}\cong k^{\oplus2n+1}$ is given by
\[
(x_{1},\ldots,x_{2n},x_{2n+1})
\mapsto
(x_{2n},\ldots,x_{1},-x_{2n+1}).
\]
We can easily check that $\theta(\varphi^{\GL_{2n+1}}_{a})=-(\varphi^{\GL_{2n+1}}_{-a})^{-1}$.
This implies that
\[
(\pi_{a,\zeta}^{\GL_{2n+1}})^{\theta}
:=
\pi_{a,\zeta}^{\GL_{2n+1}}\circ\theta
\cong
\pi_{-a,\zeta^{-1}}^{\GL_{2n+1}}.
\]
Therefore we see that a simple supercuspidal representation $\pi_{a,\zeta}^{\GL_{2n+1}}$ of $\GL_{2n+1}(F)$ with trivial central character can be $\theta$-stable (i.e., $(\pi_{a,\zeta}^{\GL_{2n+1}})^{\theta}\cong\pi_{a,\zeta}^{\GL_{2n+1}}$) only when $p=2$.
Furthermore, when $p=2$, 
\[
\{
\pi^{\GL_{2n+1}}_{a,1}
\mid
a\in k^{\times}
\}
\]
represents the set of equivalence classes of $\theta$-stable simple supercuspidal representations of $\GL_{2n+1}(F)$ with trivial central character.
(Note that the condition that $\zeta=\zeta^{-1}$ forces that $\zeta=1$ since $\zeta$ is a $(2n+1)$-th root of unity).
In the following, we write $\pi^{\GL_{2n+1}}_{a}$ instead of $\pi^{\GL_{2n+1}}_{a,1}$, for short.

\subsection{The case of $\Sp_{2n}$}\label{subsec:ssc-Sp}
Assume now $p=2$.
We consider the case of  
\[
\Sp_{2n} := \{g \in \GL_{2n} \mid {}^{t}\!gJ_{2n}g=J_{2n} \}.
\]
We have the Iwahori subgroup $I_{\GL_{2n}}$ of $\GL_{2n}$, so we can define the Iwahori subgroup $I_{\Sp_{2n}}$ of $\Sp_{2n}$ by intersection, and similarly for $I_{\Sp_{2n}}^{+}$ and $I_{\Sp_{2n}}^{++}$.
%We introduce the standard Iwahori subgroup $I_{\Sp_{2n}}$ and its Moy--Prasad filtration subgroups $I_{\Sp_{2n}}^{+}$ and $I_{\Sp_{2n}}^{++}$ in a similar way to above.
Then we have
\begin{align*}
I_{\Sp_{2n}}^{+}/I_{\Sp_{2n}}^{++} &\cong k^{\oplus n+1} \\
(y_{ij})_{ij} &\mapsto \bigr(\ol{y_{1 2}}, \ldots, \ol{y_{n, n+1}}, \ol{y_{2n, 1}\varpi^{-1}}\bigl).
\end{align*}

For $a \in k^{\times}$, we define an affine generic character $\chi^{\Sp_{2n}}_{a}\colon I_{\Sp_{2n}}^{+}\rightarrow\C^{\times}$ by
\[
 \chi^{\Sp_{2n}}_{a}(y):= \psi\bigl(\ol{y_{12}}+\cdots+\ol{y_{n-1, n}}+\ol{y_{n, n+1}}+a\ol{y_{2n, 1}\varpi^{-1}}\bigr) \text{ for $y=(y_{ij})_{ij} \in I_{\Sp_{2n}}^{+}$}.
\]
Let $\pi_{a}^{\Sp_{2n}}$ be the representation of $\Sp_{2n}(F)$ defined by
\[
\pi^{\Sp_{2n}}_{a}:=\cInd^{\Sp_{2n}(F)}_{I_{\Sp_{2n}}^{+}} \chi^{\Sp_{2n}}_{a}.
\]
Then 
\[
\{
\pi^{\Sp_{2n}}_{a}
\mid
a\in k^{\times}
\}
\]
represents the set of equivalence classes of simple supercuspidal representations of $\Sp_{2n}(F)$.

\begin{rem}
When $p\neq2$, the set of equivalence classes of simple supercuspidal representations of $\Sp_{2n}(F)$ can be represented by 
\[
\{
\pi^{\Sp_{2n}}_{\xi,\kappa,a}
\mid
\xi\in\{\pm1\},\kappa\in\{0,1\},a\in k^{\times}
\}
\]
as in \cite[Section 2.4]{Oi:2018}.
Here, $\xi$ is a sign giving the value on $-I_{2n}\in\Sp_{2n}(F)$ of the central character of the simple supercuspidal representation.
%(When $p\neq2$, we have to extend to an affine generic character of $I^{+}_{\Sp_{2n}}$ to $\pm I^{+}_{\Sp_{2n}}$ and then take the compact induction.)
When $p=2$, the center $\{\pm I_{2n}\}$ of $\Sp_{2n}(F)$ is contained in the second-step Iwahori subgroup $I_{\Sp_{2n}}^{++}$, hence the central character of any simple supercuspidal representation must be trivial.
Accordingly, the parameter ``$\xi$'' does not appear in the case where $p=2$.
On the other hand, $\kappa$ is a parameter related to affine generic characters; when $p\neq2$, there are $2(q-1)$ affine generic characters of $I^{+}_{\Sp_{2n}}$ up to equivalence.
When $p=2$, the equality $k^{\times2}=k^{\times}$ guarantees that any affine generic character is equivalent to an affine generic character of the form $\chi^{\Sp_{2n}}_{a}$ with $a\in k^{\times}$, thus also the parameter ``$\kappa$'' does not appear.
\end{rem}

\section{Characters of simple supercuspidal representations}\label{sec:char}

In this section, we compute the ($\theta$-twisted) characters of simple supercuspidal representations.
Recall that we are assuming that $p=2$.

\subsection{The case of $\Sp_{2n}$}\label{subsec:char-Sp}

We first recall the notion of the \textit{(Harish-Chandra) character}.
Let us consider a connected reductive group $\G$ over $F$.
For any irreducible admissible representation $\pi$ of $\G(F)$, we have its (Harish-Chandra) character $\Theta_{\pi}$.
This is a $\G(F)$-conjugate-invariant $\C$-valued function defined on the set of regular semisimple elements of $\G(F)$.
Any irreducible admissible representation $\pi$ is determined up to equivalence by its character $\Theta_{\pi}$ (\cite{MR0414797}).

In \cite{Oi:2018}, we computed the characters of simple supercuspidal representations of $\Sp_{2n}(F)$ at certain special elements which we call \textit{affine generic elements}:

\begin{defn}\label{defn:aff-gen}
We say that an element $y$ of $I^{+}_{\Sp_{2n}}$ is \textit{affine generic} if any component of its image in $I^{+}_{\Sp_{2n}}/I^{++}_{\Sp_{2n}}\cong k^{\oplus n+1}$ is nonzero.
\end{defn}

Since any affine generic element of $\Sp_{2n}(F)$ is regular semisimple (see \cite[Remark 3.11]{Oi:2018}), it makes sense to consider the value of the character of a simple supercuspidal representation of $\Sp_{2n}(F)$ at an affine generic element.
In fact, the same method as in \cite{Oi:2018} is still available for computing the character of a simple supercuspidal representation at affine generic elements even when $p=2$.

Before we explain our computation of the characters, let us recall the \textit{Kloosterman sum}, which is defined as follows for any $N\in\Z_{>0}$ and $x\in k^{\times}$:
\[
\mathrm{Kl}_{x}^{N}(\psi)
:=
\sum_{\begin{subarray}{c}x_{1},\ldots,x_{N} \in k^{\times} \\ x_{1}\cdots x_{N}=x\end{subarray}}
\psi(x_{1}+\cdots +x_{N}).
\]

\begin{prop}\label{prop:char-Sp}
Let $y=(y_{ij})_{ij}\in I_{\Sp_{2n}}^{+}$ be an affine generic element.
Then we have
\[
\Theta_{\pi^{\Sp_{2n}}_{a}}(y)
=
\Kl^{n+1}_{\beta}(\psi),
\]
where $\beta$ is the image of $ay_{1,2}^{2}\cdots y_{n-1,n}^{2}y_{n,n+1}y_{2n,1}\varpi^{-1}\in\mcO^{\times}$ in the residue field $k$.
\end{prop}

\begin{proof}
By the same argument as in \cite[Proposition 3.9]{Oi:2018}, but by noting that $-1=1$ in $k^{\times}$, the Frobenius formula of the Harish-Chandra character (\cite{MR1039842}) implies that 
\begin{align*}
\Theta_{\pi^{\Sp_{2n}}_{a}}(y)
&=
\sum_{t_{1},\ldots,t_{n}\in k^{\times}} \psi\biggl(\frac{t_{1}}{t_{2}}y_{1,2}+\cdots+\frac{t_{n-1}}{t_{n}}y_{n-1,n}+t_{n}^{2}y_{n,n+1}+\frac{a}{t_{1}^{2}}y_{2n,1}\varpi^{-1}\biggr)\\
&=
\sum_{\begin{subarray}{c}s_{1},\ldots,s_{n+1}\in k^{\times} \\ s_{1}^{2}\cdots s_{n-1}^{2}s_{n}s_{n+1}=\beta\end{subarray}}\psi(s_{1}+\cdots+s_{n+1}).
\end{align*}
Since $p=2$, the square map $(-)^{2}\colon k^{\times}\rightarrow k^{\times}$ is nothing but the Frobenius map and thus bijective.
Hence, by also noting that we chose $\psi$ so that $\psi(x^{2})=\psi(x)$ for any $x\in \mcO$ (or $x\in k$), we have
\begin{align*}
\sum_{\begin{subarray}{c}s_{1},\ldots,s_{n+1}\in k^{\times} \\ s_{1}^{2}\cdots s_{n-1}^{2}s_{n}s_{n+1}=\beta\end{subarray}}\psi(s_{1}+\cdots+s_{n+1})
&=\sum_{\begin{subarray}{c}s_{1},\ldots,s_{n+1}\in k^{\times} \\ s_{1}\cdots s_{n-1}s_{n}s_{n+1}=\beta\end{subarray}}\psi(s_{1}^{\frac{1}{2}}+\cdots+s_{n-1}^{\frac{1}{2}}+s_{n}+s_{n+1})\\
&=\sum_{\begin{subarray}{c}s_{1},\ldots,s_{n+1}\in k^{\times} \\ s_{1}\cdots s_{n-1}s_{n}s_{n+1}=\beta\end{subarray}}\psi(s_{1}+\cdots+s_{n-1}+s_{n}+s_{n+1})\\
&=\Kl_{\beta}^{n+1}(\psi).
\end{align*}
\end{proof}

\subsection{$\theta$-affine generic elements}\label{subsec:theta-aff-gen}
We introduce the notion of ``$\theta$-affine genericity'' for elements of Iwahori subgroups:

\begin{defn}\label{defn:theta-aff-gen}
Let $x=(x_{ij})_{ij}$ be an element of $I_{\GL_{2n+1}}^{+}\subset\GL_{2n+1}(F)$.
We say that $x$ is \textit{$\theta$-affine generic} if it satisfies
\begin{itemize}
\item
$x_{1,2}+x_{2n,2n+1}\in\mcO^{\times}$,
\item
$x_{2,3}+x_{2n-1,2n}\in\mcO^{\times}$,
\item[]
$\vdots$
\item
$x_{n,n+1}+x_{n+1,n+2}\in\mcO^{\times}$,
\item
$x_{2n+1,1}\in\mfp\smallsetminus\mfp^{2}$.
\end{itemize}
\end{defn}

\begin{lem}\label{lem:Iwahori}
Let $x=(x_{ij})_{ij}$ be an element of $I_{\GL_{2n+1}}^{+}$.
\begin{enumerate}
\item
If we let $\theta(x)=(x'_{ij})_{ij}\in I_{\GL_{2n+1}}^{+}$, then we have
\begin{itemize}
\item
$x'_{1,2}\equiv x_{2n,2n+1} \pmod \mfp$,
\item[]
$\vdots$
\item
$x'_{2n,2n+1}\equiv x_{1,2} \pmod \mfp$,
\item
$x'_{2n,1}\equiv x_{2n+1,2}-x_{1,2}\cdot x_{2n+1,1} \pmod{\mfp^{2}}$,
\item
$x'_{2n+1,2}\equiv x_{2n,1}-x_{2n,2n+1}\cdot x_{2n+1,1} \pmod{\mfp^{2}}$,
\item
$x'_{2n+1,1}\equiv -x_{2n+1,1} \pmod{\mfp^{2}}$.
\end{itemize}
\item
If we let $x\theta(x)=(z_{ij})_{ij}\in I_{\GL_{2n+1}}^{+}$, then we have
\begin{itemize}
\item
$z_{1,2}\equiv x_{1,2}+x_{2n,2n+1} \pmod \mfp$,
\item[]
$\vdots$
\item
$z_{2n,2n+1}\equiv x_{2n,2n+1}+x_{1,2} \pmod \mfp$,
\item
$z_{2n,1}\equiv x_{2n,1}+x_{2n+1,2}-(x_{1,2}+x_{2n,2n+1})\cdot x_{2n+1,1} \pmod{\mfp^{2}}$,
\item
$z_{2n+1,2}\equiv x_{2n,1}+x_{2n+1,2}\pmod{\mfp^{2}}$,
\item
$z_{2n+1,1}\equiv 0 \pmod{\mfp^{2}}$.
\end{itemize}
\end{enumerate}
\end{lem}

\begin{proof}
Let $\mathfrak{A}$ denote the standard Iwahori order of $M_{2n+1}(F)$ and $\mathfrak{Q}$ denote its radical:
\[
\mathfrak{A}
=
\begin{pmatrix}
\mcO&&\mcO\\
 &\ddots&\\
 \mfp&&\mcO
\end{pmatrix},\quad
\mathfrak{Q}
=
\begin{pmatrix}
\mfp&&\mcO\\
 &\ddots&\\
 \mfp&&\mfp
\end{pmatrix}.
\]
Let us temporarily write $\varphi$ for $\varphi_{1}^{\GL_{2n+1}}$ in this proof; note that $\mathfrak{Q}^{k}=\varphi^{k}\mathfrak{A}$.
Also note that $I_{\GL_{2n+1}}^{+}=I_{2n+1}+\mathfrak{Q}$ and $I_{\GL_{2n+1}}^{++}=I_{2n+1}+\mathfrak{Q}^{2}$.

If we write $x=I_{2n+1}+X\in I_{\GL_{2n+1}}^{+}$ with $X\in\mathfrak{Q}$, then its inverse modulo $\mathfrak{Q}^{3}$ is given by $I_{2n+1}-X+X^{2}$.
Thus, by putting $\eta(X):=J_{2n+1}{}^{t}XJ_{2n+1}^{-1}\in\mathfrak{Q}$, we get
\[
\theta(x)
=(I_{2n+1}+\eta(X))^{-1}
\equiv I_{2n+1}-\eta(X)+\eta(X)^{2}
\pmod{\mathfrak{Q}^{3}}.
\]

Let us take diagonal matrices $D,D'\in\mathfrak{A}$ such that $X\equiv\varphi D+\varphi^{2}D'\pmod{\mathfrak{Q}^{3}}$.
Then we have
\begin{align*}
\theta(x)
&\equiv I_{2n+1}-\eta(\varphi D+\varphi^{2}D')+\eta(\varphi D+\varphi^{2}D')^{2}\\
&\equiv I_{2n+1}-\eta(\varphi D)-\eta(\varphi^{2}D')+\eta(\varphi D)^{2}
\pmod{\mathfrak{Q}^{3}}.
\end{align*}
Now, by noting that the action $\eta(-)$ on any matrix $Y\in M_{2n+1}(F)$ is given by the combination of
\begin{itemize}
\item
reflecting $Y$ with respect to the second-diagonal and
\item
multiplying the $(i,j)$-entry by the sign $(-1)^{i+j}$,
\end{itemize}
it is not difficult to check the first assertion (1).

Moreover, as we have
\begin{align*}
x\theta(x)
&\equiv (1+\varphi D+\varphi^{2}D')(1-\eta(\varphi D)-\eta(\varphi^{2}D')+\eta(\varphi D)^{2})\\
&\equiv 1+\varphi D-\eta(\varphi D)-\eta(\varphi^{2}D')+\eta(\varphi D)^{2}-\varphi D\cdot \eta(\varphi D)+\varphi^{2} D'
\pmod{\mathfrak{Q}^{3}},
\end{align*}
we can also check the second assertion (2).
\end{proof}

\begin{lem}\label{lem:char-pol}
Let $x\in I_{\GL_{2n+1}}^{+}$ be a $\theta$-semisimple element with $x\theta(x)=(z_{ij})_{ij}\in I_{\GL_{2n+1}}^{+}$.
Let $p_{x,\theta}(T)\in F[T]$ be the characteristic polynomial of $x\theta(x)$, and write
\[
p_{x,\theta}(T)=(T-1)^{2n+1}+a_{2n}(T-1)^{2n}+\cdots+a_{1}(T-1)+a_{0}.
\]
Then we have
\begin{enumerate}
 \item
 $a_{i}\in\mfp$ for each $0\leq i\leq 2n$,
 \item
 $a_{0}=0$, and
 \item
 $a_{1}\equiv -z_{2,3}\cdots z_{2n-1,2n}(z_{2n, 2n+1}z_{2n+1,2}+z_{2n,1}z_{1,2}) \quad\mod\mfp^{2}$.
\end{enumerate}
\end{lem}

\begin{proof}
By Lemma \ref{lem:Iwahori} (2), the element $x\theta(x)$ belongs to the subgroup $P_{\x}^{+}$ introduced in \cite[Section 7.1]{Oi:2018}, hence the same arguments as in \cite[Lemma 7.6]{Oi:2018} can be applied to check (1) and (3).
Let us check (2).
As we assume that $x$ is $\theta$-semisimple, we can find an element $g\in\GL_{2n+1}(\ol{F})$ such that $gx\theta(g)^{-1}$ is a diagonal element, say $t=\diag(t_{1},\ldots,t_{2n+1})$.
Then we have
\[
gx\theta(x)g^{-1}
=t\theta(t)
=\diag\biggl(\frac{t_{1}}{t_{2n+1}},\ldots,\frac{t_{n}}{t_{n+2}},1,\frac{t_{n+2}}{t_{n}},\ldots,\frac{t_{2n+1}}{t_{1}}\biggr).
\]
Hence $x\theta(x)$ has $1$ as its eigenvalue.
In other words, the constant term of the characteristic polynomial $p_{x,\theta}$ of $x\theta(x)$ (with respect to $(T-1)$) is zero.
\end{proof}

\begin{rem}
We remark that, in \cite[Lemma 7.6]{Oi:2018}, we proved (2) without assuming that $x$ is $\theta$-semisimple but assuming that $p\neq2$.
\end{rem}

By combining these lemmas, we get the following:
\begin{prop}\label{prop:Eisenstein}
Let $x\in I_{\GL_{2n+1}}^{+}$ be a $\theta$-semisimple element with $x\theta(x)=(z_{ij})_{ij}\in I_{\GL_{2n+1}}^{+}$.
Then $x$ is $\theta$-affine generic if and only if the characteristic polynomial $p_{x,\theta}(T)$ of $x\theta(x)$ is of the form
\[
(T-1)\cdot(\text{Eisenstein polynomial in $(T-1)$}).
\]
Furthermore, the constant term of the Eisenstein polynomial is given by 
\[
z_{1,2}^{2}\cdots z_{n,n+1}^{2}x_{2n+1,1} \pmod{\mfp^{2}}.
\]
\end{prop}

\begin{proof}
By Lemma \ref{lem:char-pol}, $p_{x,\theta}(T)$ is given by the product of $(T-1)$ and a polynomial $(T-1)^{2n}+a_{2n}(T-1)^{2n-1}+\cdots+a_{1}$ such that
\begin{itemize}
\item
$a_{i}\in\mfp$ for each $1\leq i\leq 2n$, and
\item
$a_{1}\equiv -z_{2,3}\cdots z_{2n-1,2n}(z_{2n, 2n+1}z_{2n+1,2}+z_{2n,1}z_{1,2}) \quad\mod\mfp^{2}$.
\end{itemize}
The latter polynomial is Eisenstein if and only if 
\[
-z_{2,3}\cdots z_{2n-1,2n}(z_{2n, 2n+1}z_{2n+1,2}+z_{2n,1}z_{1,2}) \not\equiv0\pmod{\mfp^{2}}.
\]
Since $z_{1,2}\ldots,z_{2n,2n+1}\in\mcO$ and $z_{2n+1,2},z_{2n,1}\in\mfp$, this is furthermore equivalent to that $z_{2,3}\ldots,z_{2n-1,2n}\in\mcO^{\times}$ and $z_{2n, 2n+1}z_{2n+1,2}+z_{2n,1}z_{1,2} \not\equiv0\pmod{\mfp^{2}}$.
Let us investigate the latter condition.
By Lemma \ref{lem:Iwahori}, we have $z_{1,2}\equiv z_{2n,2n+1} \pmod{\mfp}$ and $z_{2n,1}\equiv z_{2n+1,2}-z_{1,2}\cdot x_{2n+1,1}\pmod{\mfp^{2}}$.
Thus we get
\begin{align*}
z_{2n, 2n+1}z_{2n+1,2}+z_{2n,1}z_{1,2}
&\equiv z_{1,2}(z_{2n+1,2}+z_{2n,1})\\
&\equiv z_{1,2}(z_{2n+1,2}+z_{2n+1,2}-z_{1,2}\cdot x_{2n+1,1})\\
&\equiv -z_{1,2}^{2}\cdot x_{2n+1,1} \pmod{\mfp^{2}}.
\end{align*}
We have $-z_{1,2}^{2}\cdot x_{2n+1,1}\not\equiv0 \pmod{\mfp^{2}}$ if and only if we have $z_{1,2}\,(=z_{2n,2n+1}) \in\mcO^{\times}$ and $x_{2n+1,1}\in\mfp\smallsetminus\mfp^{2}$.
This completes the proof.
\end{proof}

Recall that we say that an element $x\in\GL_{2n+1}$ is \textit{$\theta$-semisimple} if $x$ is $\theta$-conjugate to a diagonal element of $\GL_{2n+1}$ over $\overline{F}$ and that a $\theta$-semisimple element $x\in\GL_{2n+1}$ is \textit{strongly $\theta$-regular} if its $\theta$-centralizer
\[
Z_{\GL_{2n+1}}(x\rtimes\theta)
:=\{g\in\GL_{2n+1}\mid gx\theta(g)^{-1}=x\}
\]
is abelian.

\begin{lem}\label{lem:regular}
Any $\theta$-affine generic element $x\in I_{\GL_{2n+1}}^{+}$ is strongly $\theta$-regular $\theta$-semisimple.
\end{lem}

\begin{proof}
The $\theta$-semisimplicity of the element $x\in \GL_{2n+1}$ is equivalent to the semisimplicity of the element $x\rtimes\theta$ in the disconnected reductive group $\GL_{2n+1}\rtimes\langle\theta\rangle$.
In order to show that $x\rtimes\theta\in\GL_{2n+1}\rtimes\langle\theta\rangle$ is semisimple, it suffices to show that $x\theta(x)\in\GL_{2n+1}$ is semisimple.
Indeed, if we take the Jordan decomposition $x\rtimes\theta=su$ with semisimple $s\in\GL_{2n+1}\rtimes\langle\theta\rangle$ and unipotent $u\in\GL_{2n+1}\rtimes\langle\theta\rangle$, then we have 
\[
x\theta(x)
=(x\rtimes\theta)^{2}
=susu
=s^{2}u^{2}
\]
(note that $s$ and $u$ commutes as $x\rtimes\theta=su$ is the Jordan decomposition).
Thus, if this element is semisimple, then the unipotent part $u^{2}$ is necessarily trivial.
Since $F$ is of characteristic $0$, this implies that $u$ is trivial, hence $x\rtimes\theta$ is semisimple.

Therefore, as $x\theta(x)$ is (strongly) regular semisimple by Proposition \ref{prop:Eisenstein}, we see that $x$ is $\theta$-semisimple.
Furthermore, by noting that $Z_{\GL_{2n+1}}(x\rtimes\theta)$ is contained in the usual centralizer $Z_{\GL_{2n+1}}(x\theta(x))$ of $x\theta(x)$, we see that $x$ is strongly $\theta$-regular since $x\theta(x)$ is strongly regular semisimple.
\end{proof}

\begin{example}\label{example}
The following element of $I_{\GL_{2n+1}}^{+}\subset \GL_{2n+1}(F)$ is (obviously from the definition) $\theta$-affine generic for any $u\in k^{\times}$:
\[\label{gu}
g_{u}:=
\left(
\begin{array}{cccc:c:cccc} 
1&1&&&&&&&\\
&\ddots&\ddots&&&&&&\\
&&\ddots&1&&&&&\\ 
&&&1&1&&&&\\ \hdashline
&&&&1&0&&&\\ \hdashline
&&&&&1&0&&\\
&&&&&&\ddots&\ddots&\\
&&&&&&&\ddots&0\\
\varpi u&&&&&&&&1
\end{array}
\right).
\tag{$\dagger$}
\]
Here, the middle dashed row and column denote the $(n+1)$-th row and $(n+1)$-th column, respectively.
We can check that
\[
\theta(g_{u})
=
\left(
\begin{array}{cccc:c:cccc} 
1&0&&&&&&&\\
&\ddots&\ddots&&&&&&\\
&&\ddots&0&&&&&\\ 
&&&1&0&&&&\\ \hdashline
-\varpi u&&&&1&1&\cdots&\cdots&1\\ \hdashline
-\varpi u&&&&&1&1&\cdots&1\\
\vdots&&&&&&\ddots&\ddots&\vdots\\
\vdots&&&&&&&\ddots&1\\
-\varpi u&&&&&&&&1
\end{array}
\right)
\]
and that
\[
g_{u}\theta(g_{u})
=
\left(
\begin{array}{cccc:c:cccc} 
1&1&&&&&&&\\
&\ddots&\ddots&&&&&&\\
&&\ddots&1&&&&&\\ 
-\varpi u&&&1&1&1&\cdots&\cdots&1\\ \hdashline
-\varpi u&&&&1&1&\cdots&\cdots&1\\ \hdashline
-\varpi u&&&&&1&1&\cdots&1\\
\vdots&&&&&&\ddots&\ddots&\vdots\\
-\varpi u&&&&&&&\ddots&1\\
0&&&&&&&&1
\end{array}
\right).
\]
We define $h_{u}\in \GL_{2n}(F)$ to be the upper-left $2n$-by-$2n$ minor matrix of $g_{u}\theta(g_{u})$.
Thus, if define $n$-by-$n$ matrices $P$, $X$, $Y$, and $Q$ by
\[
P:=
\begin{pmatrix}
1&1&&\\
&\ddots&\ddots&\\
&&\ddots&1\\
-\varpi u&&&1
\end{pmatrix},\quad
X:=
\begin{pmatrix}
0&\cdots&0\\
\vdots&&\vdots\\
0&\cdots&0\\
1&\cdots&1
\end{pmatrix},
\]
\[
Y:=
\begin{pmatrix}
-\varpi u&0&\cdots&0\\
\vdots&\vdots&&\vdots\\
-\varpi u&0&\cdots&0
\end{pmatrix},\quad
Q:=
\begin{pmatrix}
1&\cdots&1\\
&\ddots&\vdots\\
&&1
\end{pmatrix}
\]
then we have 
\[\label{hu}
h_{u}=
\begin{pmatrix}
P&X\\
Y&Q
\end{pmatrix}.
\tag{$\dagger\dagger$}
\]
Then we observe that $h_{u}$ belongs to the pro-$p$ Iwahori subgroup $I_{\GL_{2n}}^{+}$ of $\GL_{2n}(F)$ and is affine generic.
\end{example}

\begin{lem}\label{lem:g-h}
The element $h_{u}$ as in Example \ref{example} belongs to $\Sp_{2n}(F)$.
\end{lem}

\begin{proof}
%\textcolor{blue}{MO: Maybe this proof is over-explaining, so eventually we should reduce it. But I want to record my computation in this draft version for our convenience.}
%
Our task is to show that ${}^{t}h_{u}J_{2n}h_{u}=J_{2n}$.
Since $J_{2n}=\begin{pmatrix}&J_{n}\\(-1)^{n}J_{n}&\end{pmatrix}$, we have
\begin{align*}
{}^{t}h_{u}J_{2n}h_{u}
&=
\begin{pmatrix}
{}^{t}P&{}^{t}Y\\
{}^{t}X&{}^{t}Q
\end{pmatrix}
\begin{pmatrix}
&J_{n}\\
(-1)^{n}J_{n}&
\end{pmatrix}
\begin{pmatrix}
P&X\\
Y&Q
\end{pmatrix}\\
&=
\begin{pmatrix}
{}^{t}P&{}^{t}Y\\
{}^{t}X&{}^{t}Q
\end{pmatrix}
\begin{pmatrix}
J_{n}Y&J_{n}Q\\
(-1)^{n}J_{n}P&(-1)^{n}J_{n}X
\end{pmatrix}\\
&=
\begin{pmatrix}
{}^{t}PJ_{n}Y-{}^{t}Y{}^{t}J_{n}P&{}^{t}PJ_{n}Q-{}^{t}Y{}^{t}J_{n}X\\
{}^{t}XJ_{n}Y-{}^{t}Q{}^{t}J_{n}P&{}^{t}XJ_{n}Q-{}^{t}Q{}^{t}J_{n}X
\end{pmatrix}
\end{align*}
(note that ${}^{t}J_{n}=(-1)^{n-1}J_{n}$).
Hence it is enough to check that
\begin{enumerate}
\item
${}^{t}PJ_{n}Y-{}^{t}Y{}^{t}J_{n}P=0$,
\item
${}^{t}PJ_{n}Q-{}^{t}Y{}^{t}J_{n}X=J_{n}$,
\item
${}^{t}XJ_{n}Y-{}^{t}Q{}^{t}J_{n}P=(-1)^{n}J_{n}$, and
\item
${}^{t}XJ_{n}Q-{}^{t}Q{}^{t}J_{n}X=0$.
\end{enumerate}

We first consider (1).
As ${}^{t}Y{}^{t}J_{n}P={}^{t}({}^{t}PJ_{n}Y)$, it suffices to show that ${}^{t}PJ_{n}Y$ is symmetric.
We can easily compute ${}^{t}PJ_{n}Y$ as follows:
\begin{align*}
&
\begin{pmatrix}
1&&&-\varpi u\\
1&\ddots&&\\
&\ddots&\ddots&\\
&&1&1
\end{pmatrix}
\begin{pmatrix}
&&&1\\
&&-1&\\
&\adots&&\\
(-1)^{n-1}&&&
\end{pmatrix}
\begin{pmatrix}
-\varpi u&0&\cdots&0\\
\vdots&\vdots&&\vdots\\
-\varpi u&0&\cdots&0
\end{pmatrix}\\
&=
\begin{pmatrix}
(-1)^{n}\varpi u&&&1\\
&&-1&1\\
&\adots&\adots&\\
(-1)^{n-1}&(-1)^{n-2}&&
\end{pmatrix}
\begin{pmatrix}
-\varpi u&0&\cdots&0\\
\vdots&\vdots&&\vdots\\
-\varpi u&0&\cdots&0
\end{pmatrix}\\
&=
\begin{pmatrix}
(-1)^{n+1}(\varpi u)^{2}-\varpi u&0&\cdots&0\\
0&0&\cdots&0\\
\vdots&\vdots&\ddots&\vdots\\
0&0&\cdots&0
\end{pmatrix}.
\end{align*}

Similarly, we can show (4) by noting that ${}^{t}Q{}^{t}J_{n}X={}^{t}({}^{t}XJ_{n}Q)$ and that ${}^{t}X{}^{t}J_{n}Q$ is a symmetric matrix, which can be checked as follows:
\begin{align*}
&
\begin{pmatrix}
0&\cdots&0&1\\
\vdots&&\vdots&\vdots\\
0&\cdots&0&1
\end{pmatrix}
\begin{pmatrix}
&&&1\\
&&-1&\\
&\adots&&\\
(-1)^{n-1}&&&
\end{pmatrix}
\begin{pmatrix}
1&\cdots&1\\
&\ddots&\vdots\\
&&1
\end{pmatrix}
\\
&=
\begin{pmatrix}
0&\cdots&0&1\\
\vdots&&\vdots&\vdots\\
0&\cdots&0&1
\end{pmatrix}
\begin{pmatrix}
&&&1\\
&&-1&-1\\
&\adots&&\vdots\\
(-1)^{n-1}&\cdots&\cdots&(-1)^{n-1}
\end{pmatrix}\\
&=
(-1)^{n-1}
\begin{pmatrix}
1&\cdots&1\\
\vdots&\ddots&\vdots\\
1&\cdots&1
\end{pmatrix}.
\end{align*}

Finally, let us consider (2) and (3).
Since the equality (3) is obtained by transposing the equality (2), it suffices to show (2).
We have
\begin{align*}
{}^{t}PJ_{n}Q
&=
\begin{pmatrix}
(-1)^{n}\varpi u&&&1\\
&&-1&1\\
&\adots&\adots&\\
(-1)^{n-1}&(-1)^{n-2}&&
\end{pmatrix}
\begin{pmatrix}
1&\cdots&1\\
&\ddots&\vdots\\
&&1
\end{pmatrix}\\
&=
\begin{pmatrix}
(-1)^{n}\varpi u&\cdots&(-1)^{n}\varpi u&(-1)^{n}\varpi u+1\\
&&-1&\\
&\adots&&\\\
(-1)^{n-1}&&&
\end{pmatrix}
\end{align*}
and
\begin{align*}
{}^{t}Y{}^{t}J_{n}X
&=
\begin{pmatrix}
-\varpi u&\cdots&-\varpi u\\
0&\cdots&0\\
\vdots&&\vdots\\
0&\cdots&0
\end{pmatrix}
\begin{pmatrix}
&&(-1)^{n-1}\\
&\adots&\\
1&&
\end{pmatrix}
\begin{pmatrix}
0&\cdots&0\\
\vdots&&\vdots\\
0&\cdots&0\\
1&\cdots&1
\end{pmatrix}\\
&=
\begin{pmatrix}
-\varpi u&\cdots&-\varpi u\\
0&\cdots&0\\
\vdots&&\vdots\\
0&\cdots&0
\end{pmatrix}
\begin{pmatrix}
(-1)^{n-1}&\cdots&(-1)^{n-1}\\
0&\cdots&0\\
\vdots&&\vdots\\
0&\cdots&0
\end{pmatrix}\\
&=
\begin{pmatrix}
(-1)^{n}\varpi u&\cdots&(-1)^{n}\varpi u\\
0&\cdots&0\\
\vdots&&\vdots\\
0&\cdots&0
\end{pmatrix}.
\end{align*}
Hence we get ${}^{t}PJ_{n}Q-{}^{t}Y{}^{t}J_{n}X=J_{n}$.
\end{proof}

\subsection{Twisted characters at $\theta$-affine generic elements}\label{subsec:theta-char-GL}
We next recall the notion of the \textit{$\theta$-twisted (Harish-Chandra) character}.
Let $\pi$ be a $\theta$-stable (i.e., $\pi^{\theta}:=\pi\circ\theta\cong\pi$) irreducible admissible representation of $\GL_{2n+1}(F)$.
By fixing an isomorphism $I\colon\pi\cong\pi^{\theta}$, we have the $\theta$-twisted character $\Theta_{\pi,\theta}$ of $\pi$.
This is a $\C$-valued function defined on the set of $\theta$-regular $\theta$-semisimple elements of $\G(F)$.
Similarly to the usual character, the $\theta$-twisted character is invariant under $\theta$-conjugation by $\G(F)$.
Any $\theta$-stable irreducible admissible representation $\pi$ is determined up to equivalence by its $\theta$-twisted character (\cite{MR3632513}).

The aim of this subsection is to compute the $\theta$-twisted characters of the $\theta$-stable simple supercuspidal representation $\pi^{\GL_{2n+1}}_{a}$ of $\GL_{2n+1}(F)$ at $\theta$-affine generic elements.
For this, let us first specify the choice of an intertwiner $I\colon\pi^{\GL_{2n+1}}_{a}\cong\pi_{a}^{\GL_{2n+1},\theta}$.
Recall that $\pi^{\GL_{2n+1}}_{a}$ is defined to be the compact induction of a character $\tilde{\chi}^{\GL_{2n+1}}_{a}$ from the subgroup $ZI_{\GL_{2n+1}}^{+}\lan\varphi_{a^{-1}}^{\GL_{2n+1}}\rangle$.
By noting that the subgroup $ZI_{\GL_{2n+1}}^{+}\lan\varphi_{a^{-1}}^{\GL_{2n+1}}\rangle$ is $\theta$-stable and that the character $\tilde{\chi}^{\GL_{2n+1}}_{a}$ is $\theta$-invariant, we have a canonical isomorphism
\[
\cInd^{\GL_{2n+1}(F)}_{ZI_{\GL_{2n+1}}^{+}\lan\varphi_{a^{-1}}^{\GL_{2n+1}}\rangle} \tilde{\chi}^{\GL_{2n+1}}_{a}
\cong
\bigl(\cInd^{\GL_{2n+1}(F)}_{ZI_{\GL_{2n+1}}^{+}\lan\varphi_{a^{-1}}^{\GL_{2n+1}}\rangle} \tilde{\chi}^{\GL_{2n+1}}_{a}\bigr)^{\theta},
\]
which is given by $f\mapsto f\circ\theta$ explicitly.
We adopt this isomorphism as our intertwiner $I\colon\pi^{\GL_{2n+1}}_{a}\cong\pi_{a}^{\GL_{2n+1},\theta}$.
Accordingly, we get the $\theta$-twisted character $\Theta_{\pi^{\GL_{2n+1}}_{a},\theta}$ of $\pi^{\GL_{2n+1}}_{a}$ normalized with respect to $I$.

Recall that any $\theta$-affine generic element of $\GL_{2n+1}(F)$ is strongly $\theta$-regular $\theta$-semisimple by Lemma \ref{lem:regular}.
Hence it makes sense to consider the value of the $\theta$-twisted character $\Theta_{\pi^{\GL_{2n+1}}_{a},\theta}$ at a $\theta$-affine generic element.
The following proposition is the key to our computation of the $\theta$-twisted character.

\begin{prop}\label{prop:theta-norm}
Let $x=(x_{ij})_{ij}\in I_{\GL_{2n+1}}^{+}$ be a $\theta$-affine generic element.
If $g\in \GL_{2n+1}(F)$ satisfies $gx\theta(g)^{-1}\in ZI_{\GL_{2n+1}}^{+}\lan\varphi_{a^{-1}}^{\GL_{2n+1}}\ran$, then $g$ belongs to $ZT^{\theta}(q)I_{\GL_{2n+1}}^{+}\lan\varphi_{a^{-1}}^{\GL_{2n+1}}\ran$, where
\[
T^{\theta}(q):=\{\diag(t_{1},\ldots,t_{n},1,t_{n}^{-1},\ldots,t_{1}^{-1}) \mid t_{i}\in k^{\times}\}.
\]
\end{prop}

\begin{proof}
In this proof, we write $\tilde{I}^{+}_{\GL_{2n+1},a^{-1}}$ instead of $ZI_{\GL_{2n+1}}^{+}\lan\varphi_{a^{-1}}^{\GL_{2n+1}}\ran$, for short.

Let us take a $\theta$-affine generic element $x=(x_{ij})_{ij}\in I_{\GL_{2n+1}}^{+}$.
We write $x\theta(x)=(z_{ij})_{ij}$.
Then, by the definition of the $\theta$-affine genericity, Lemma \ref{lem:Iwahori} (2) implies that
\begin{itemize}
\item
$z_{i,i+1}$ belongs to $\mcO^{\times}$ for any $1\leq i\leq 2n$, and
\item
at least one of $z_{2n,1}\in\mfp$ or $z_{2n+1,2}\in\mfp$ does not belong to $\mfp^{2}$.
\end{itemize}
Therefore, we see that at least one of 
\begin{itemize}
\item
the upper-left $2n$-by-$2n$ minor matrix $(z_{ij})_{1\leq i,j\leq2n} \in I_{\GL_{2n}}^{+}$ of $x\theta(x)$ or
\item
the lower-right $2n$-by-$2n$ minor matrix $(z_{ij})_{2\leq i,j\leq2n+1} \in I_{\GL_{2n}}^{+}$ of $x\theta(x)$
\end{itemize}
is an affine generic element of the pro-$p$ Iwahori subgroup $I_{\GL_{2n}}^{+}$ of $\GL_{2n}(F)$.
Since the following proof can proceed with the same argument in both cases, we consider only the former case; let us suppose that the upper-left $2n$-by-$2n$ minor matrix $x':=(z_{ij})_{1\leq i,j\leq2n} \in I_{\GL_{2n}}^{+}$ of $x\theta(x)$ is affine generic.

Let $g\in \GL_{2n+1}(F)$ be an element satisfying $gx\theta(g)^{-1}\in \tilde{I}^{+}_{\GL_{2n+1},a^{-1}}$.
Note that the subgroup $\tilde{I}^{+}_{\GL_{2n+1},a^{-1}}$ is stable under the $\theta$-conjugation of $\tilde{I}^{+}_{\GL_{2n+1},a^{-1}}$.
Furthermore, the $\theta$-affine genericity of $g$ is preserved by $\theta$-conjugation of any element of $\tilde{I}^{+}_{\GL_{2n+1},a^{-1}}$ (this can be easily seen by Proposition \ref{prop:Eisenstein}).
Therefore, in order to show the assertion, we may freely replace $g$ with any element in its $\tilde{I}^{+}_{\GL_{2n+1},a^{-1}}$-double coset.

Recall that the double cosets $I_{\GL_{2n+1}}^{+}\backslash \GL_{2n+1}(F)/I_{\GL_{2n+1}}^{+}$ can be described in terms of the Iwahori--Weyl group (see \cite[Section 2.1]{Oi:2018}); in this case, we may choose a representative of any $I_{\GL_{2n+1}}^{+}$-double coset to be an element of $T\cdot \mathfrak{S}_{2n+1}$.
Here, $T$ denotes the subgroup of diagonal matrices of $\GL_{2n+1}(F)$ and $\mathfrak{S}_{2n+1}$ denotes the permutation group of size $2n+1$ which is realized in $\GL_{2n+1}$ in a standard way, as permutation matrices (hence $\mathfrak{S}_{2n+1}$ normalizes $T$).
Note that the $\mathfrak{S}_{2n+1}$-part of the element $\varphi^{\GL_{2n+1}}_{a^{-1}}$ is given by a cyclic permutation of length $(2n+1)$.
Hence, by replacing $g$ furthermore with an element in its $\tilde{I}^{+}_{\GL_{2n+1},a^{-1}}$-double coset, we may assume that $g$ belongs to $T\cdot\mathfrak{S}_{2n}$.
Here, $\mathfrak{S}_{2n}$ is the subgroup of $\mathfrak{S}_{2n+1}$ stabilizing the letter ``$2n+1$'', which is embedded in the upper-left $2n$-by-$2n$-part of $\GL_{2n+1}$.

Since $gx\theta(g)^{-1}$ belongs to $\tilde{I}^{+}_{\GL_{2n+1},a^{-1}}$, so does 
\[
gx\theta(g)^{-1}\cdot \theta(gx\theta(g)^{-1})
=gx\theta(x)g^{-1}.
\]
By noting that $gx\theta(x)g^{-1}$ is pro-unipotent, $gx\theta(x)g^{-1}$ belongs to, in particular, $I_{\GL_{2n+1}}^{+}$.
We write $g=\diag(g',g'')$ with $g'\in\GL_{2n}(F)$ and $g''\in \GL_{1}(F)=F^{\times}$.
Then, from $gx\theta(x)g^{-1}\in I_{\GL_{2n+1}}^{+}$, we see that $g'x'g'^{-1}\in I_{\GL_{2n}}^{+}$.
Since $x'$ is an affine generic element of $I_{\GL_{2n}}^{+}$ as observed above, we can utilize \cite[Proposition 3.4]{MR3904769} to conclude that $g'$ must belong to $Z_{\GL_{2n}}I_{\GL_{2n}}\lan\varphi^{\GL_{2n}}_{1}\ran$, where $Z_{\GL_{2n}}$ denotes the center of $\GL_{2n}(F)$.
In particular, we may write 
\begin{align}\label{condition1}
g'=\diag(\varpi^{l},\ldots,\varpi^{l})t'h'\varphi^{\GL_{2n},s}_{1}
\end{align}
with $l\in\Z$, $t'=\diag(t_{1},\ldots,t_{2n})$ ($t_{i}\in k^{\times}$), $h'\in I_{\GL_{2n}}^{+}$, and $s\in\Z$ satisfying $0\leq s<2n$.
Let us take $t_{2n+1}\in k^{\times}$ and $r'\in\Z$ such that 
\begin{align}\label{condition2}
g''\in t_{2n+1}\varpi^{r'}\mcO^{\times}.
\end{align}
Then \eqref{condition1} and \eqref{condition2} imply that we have
\begin{align*}
\tilde{I}_{\GL_{2n+1}}g\tilde{I}_{\GL_{2n+1}}
&=\tilde{I}_{\GL_{2n+1}}\diag(g',g'')\tilde{I}_{\GL_{2n+1}}\\
&=\tilde{I}_{\GL_{2n+1}}t\cdot\diag(\varpi^{l},\ldots,\varpi^{l},\varpi^{r'})\cdot\diag(\varphi^{\GL_{2n}}_{1},1)^{s}\tilde{I}_{\GL_{2n+1}}\\
&=\tilde{I}_{\GL_{2n+1}}t\cdot\diag(1,\ldots,1,\varpi^{r'-l})\cdot\diag(\varphi^{\GL_{2n}}_{1},1)^{s}\tilde{I}_{\GL_{2n+1}},
\end{align*}
where $t:=\diag(t_{1},\ldots,t_{2n+1})$.
Therefore, in the following, we assume that $g$ itself equals $t\cdot\diag(1,\ldots,1,\varpi^{r})\cdot\diag(\varphi^{\GL_{2n}}_{1},1)^{s}$, where we put $r:=r'-l$.

%Let us write $s=2n\cdot s_{1}+s_{2}$ with $s_{1},s_{2}\in\Z$ such that $0\leq s_{2}<2n$.
%Then we have
%\[
%t\cdot\diag(1,\ldots,1,\varpi^{r})\cdot\diag(\varphi^{\GL_{2n}}_{1},1)^{s}
%=
%t\cdot\diag(\varpi^{s_{1}},\ldots,\varpi^{s_{1}},\varpi^{r})\cdot\diag(\varphi^{\GL_{2n}}_{1},1)^{s_{2}}.
%\]
%This implies that $g$ and $t\cdot\diag(1,\ldots,1,\varpi^{r-s_{1}})\cdot\diag(\varphi^{\GL_{2n}}_{1},1)^{s_{2}}$ belongs to the same $ZI_{\GL_{2n+1}}^{+}\lan\varphi_{a^{-1}}^{\GL_{2n+1}}\ran$-double coset.
%Thus, by replacing $r$ with $r-s_{1}$ and $s$ with $s_{2}$, respectively, we may furthermore assume that $0\leq s<2n$ in the following.

By a simple computation, we can see that the $(2n+1)$-th column of $gx\theta(x)g^{-1}$ is given by
\[
\begin{pmatrix}
\varpi^{-r}z_{s+1,2n+1}t_{1}t_{2n+1}^{-1}\\
\vdots\\
\varpi^{-r}z_{2n,2n+1}t_{2n-s}t_{2n+1}^{-1}\\
\varpi^{-r+1}z_{1,2n+1}t_{2n+1-s}t_{2n+1}^{-1}\\
\vdots\\
\varpi^{-r+1}z_{s,2n+1}t_{2n}t_{2n+1}^{-1}\\
z_{2n+1,2n+1}
\end{pmatrix}.
\]
Here, as we have $gx\theta(x)g^{-1}\in I^{+}_{\GL_{2n+1}}$, the $(2n-s)$-th entry of this column $\varpi^{-r}z_{2n,2n+1}t_{2n-s}t_{2n+1}^{-1}$ must belong to $\mcO$.
(Recall that $0\leq s<2n$, hence $1\leq2n-s\leq 2n$).
Since $z_{2n,2n+1}$ belongs to $\mcO^{\times}$, this implies that we must have $r\leq0$.

As $gx\theta(g)^{-1}$ equals
\[
t\diag(1,\ldots,1,\varpi^{r})\diag(\varphi^{\GL_{2n}}_{1},1)^{s}x \diag(1,-\varphi^{\GL_{2n}}_{1})^{s}\diag(\varpi^{r},1,\ldots,1)\theta(t)^{-1},
\]
we have
\[
\val\circ\det(gx\theta(g)^{-1})=2r+2s.
\]
Hence the condition that $gx\theta(g)^{-1}\in ZI_{\GL_{2n+1}}^{+}\lan\varphi_{a^{-1}}^{\GL_{2n+1}}\ran$ implies that
\[
gx\theta(g)^{-1}\in I_{\GL_{2n+1}}\varphi_{a^{-1}}^{\GL_{2n+1},2r+2s}.
\tag{$\ast$}
\]
We can easily check that the $(2n+1,1)$-entry of $gx\theta(g)^{-1}$ is given by $\varpi^{2r}x_{2n+1,1}t_{2n+1}^{2}$.
By recalling that $x_{2n+1,1}$ belongs to $\mfp\smallsetminus\mfp^{2}$, we see that valuation of the $(2n+1,1)$-entry of $gx\theta(g)^{-1}$ is given by $2r+1$.
On the other hand, in general, the $(2n+1,1)$-th entry of any element of $I_{\GL_{2n+1}}\varphi_{a^{-1}}^{\GL_{2n+1},k}$ must belong to $\mfp^{\lfloor\frac{k-2}{2n+1}\rfloor+2}$ (for any $k\in\Z$):
\[
    \begin{tabular}{|c|c|c|c|c|c|c|c|c|} \hline
     $I_{\GL_{2n+1}}\varphi_{a^{-1}}^{\GL_{2n+1},k}$ & $k=0$ & $k=1$ & $k=2$ & $\cdots$ & $k=2n+2$ & $k=2n+3$ & $\cdots$ \\ \hline
     $(2n+1,1)$-entry & $\mfp$ & $\mfp$ & $\mfp^{2}$ & $\cdots$ & $\mfp^{2}$ & $\mfp^{3}$ & $\cdots$\\ \hline
    \end{tabular}
\]

Hence, so that the condition $(\ast)$ holds, we necessarily have
\[
2r+1\geq \Bigl\lfloor\frac{(2r+2s)-2}{2n+1}\Bigr\rfloor+2.
\]
In particular, we must have
\[
2r+1>\Bigl(\frac{(2r+2s)-2}{2n+1}-1\Bigr)+2,
\]
which is equivalent to the inequality $2rn>s-1$.
By recalling that $r\in\Z_{\leq0}$ and $0\leq s\leq2n$, we see that this holds only when $r=0$ and $s=0$.

Therefore we conclude that only elements $g$ of $ZT(q)I_{\GL_{2n+1}}^{+}\lan\varphi_{a^{-1}}^{\GL_{2n+1}}\ran$ can satisfy the condition $gx\theta(g)^{-1}\in \tilde{I}_{\GL_{2n+1}}$.
Finally, by looking at the diagonal entries, we easily see that such $g$ must belong to $ZT^{\theta}(q)I_{\GL_{2n+1}}^{+}\lan\varphi_{a^{-1}}^{\GL_{2n+1}}\ran$.
\end{proof}

\begin{prop}\label{prop:twisted character}
Let $x=(x_{ij})_{ij}\in I_{\GL_{2n+1}}^{+}$ be a $\theta$-affine generic element with $x\theta(x)=(z_{ij})_{ij}$.
Then we have
\[
\Theta_{\pi^{\GL_{2n+1}}_{a},\theta}(x)
=
\Kl^{n+1}_{\alpha}(\psi),
\]
where $\alpha$ is the image of $az_{1,2}^{2}\cdots z_{n,n+1}^{2}x_{2n+1,1}\varpi^{-1}\in\mcO^{\times}$ in the residue field $k$.
\end{prop}

\begin{proof}
By the Frobenius formula for the $\theta$-twisted character (\cite[I.6.2 Th\'eor\`eme]{MR3632513}), for any $\theta$-regular $\theta$-semisimple element $x\in\GL_{2n+1}(F)$, the $\theta$-twisted character $\Theta_{\pi^{\GL_{2n+1}}_{a},\theta}(x)$ normalized with respect to the intertwiner taken as above is given by
\[
\Theta_{\pi^{\GL_{2n+1}}_{a},\theta}(x)
=
\sum_{\begin{subarray}{c} g\in G/ZI_{\GL_{2n+1}}^{+}\lan\varphi_{a^{-1}}^{\GL_{2n+1}}\ran \\gx\theta(g)^{-1}\in ZI_{\GL_{2n+1}}^{+}\lan\varphi_{a^{-1}}^{\GL_{2n+1}}\ran\end{subarray}}
\tilde{\chi}^{\GL_{2n+1}}_{a}(gx\theta(g)^{-1})
\]
as long as the sum is finite (cf.\ \cite[Section 3.1]{MR3904769}).

Therefore, when $x=(x_{ij})_{ij}\in I_{\GL_{2n+1}}^{+}$ is a $\theta$-affine generic element with $x\theta(x)=(z_{ij})_{ij}$, Proposition \ref{prop:theta-norm} implies that 
\begin{align*}
\Theta_{\pi^{\GL_{2n+1}}_{a},\theta}(x)
&=\sum_{t\in T^{\theta}(q)}\tilde{\chi}^{\GL_{2n+1}}_{a}(tx\theta(t)^{-1})\\
&=\sum_{t_{1},\ldots,t_{n}\in k^{\times}}\psi\Bigl(\frac{t_{1}}{t_{2}}x_{1,2}+\frac{t_{2}}{t_{3}}x_{2,3}+\cdots+\frac{t_{2}^{-1}}{t_{1}^{-1}}x_{2n,2n+1}+\frac{t_{1}^{-1}}{t_{1}}ax_{2n+1,1}\varpi^{-1}\Bigr)\\
&=\sum_{t_{1},\ldots,t_{n}\in k^{\times}}\psi\Bigl(\frac{t_{1}}{t_{2}}z_{1,2}+\cdots+\frac{t_{n-1}}{t_{n}}z_{n-1,n}+t_{n}z_{n,n+1}+t_{1}^{-2}ax_{2n+1,1}\varpi^{-1}\Bigr)\\
&=\sum_{\begin{subarray}{c}s_{1},\ldots,s_{n},s_{n+1}\in k^{\times} \\ s_{1}^{2}\cdots s_{n}^{2}s_{n+1}=\alpha \end{subarray}}\psi(s_{1}+\cdots+s_{n-1}+s_{n}+s_{n+1}).
\end{align*}
Here, in the third equality, we used Lemma \ref{lem:Iwahori} (2).
For the same reason as in Proposition \ref{prop:char-Sp}, this equals $\Kl^{n+1}_{\alpha}(\psi)$.
\end{proof}

\section{Endoscopic lifting of simple supercuspidal representations}\label{sec:descent}

\subsection{Local Langlands correspondence for $\Sp_{2n}$}\label{subsec:LLC}
For any split connected reductive group $\G$ over $F$, we let $\hat{\G}$ denote the Langlands dual group.
We say that a homomorphism $\phi\colon W_{F}\times\SL_{2}(\C)\rightarrow \hat\G$ is an \textit{$L$-parameter} of $\G$ if $\phi$ is smooth on $W_{F}$ and the restriction $\phi|_{\SL_{2}(\C)}\colon\SL_{2}(\C)\rightarrow\hat{\G}$ is algebraic.

Recall that the Langlands dual group of $\Sp_{2n}$ is given by $\SO_{2n+1}(\C)$.
Hence, when $\G=\Sp_{2n}$, an $L$-parameter of $\G$ is nothing but a $(2n+1)$-dimensional self-dual orthogonal representation of $W_{F}\times\SL_{2}(\C)$.
Moreover, in fact, two $L$-parameters of $\Sp_{2n}$ are conjugate under $\SO_{2n+1}(\C)$ if and only if they are conjugate under $\GL_{2n+1}(\C)$, in which $\SO_{2n+1}(\C)$ is embedded (\cite[Theorem 8.1 (ii)]{MR3202556}).
Thus the $\hat{\G}$-conjugacy class of an $L$-parameter of $\G$ is nothing but the isomorphism class of a $(2n+1)$-dimensional self-dual orthogonal representation of $W_{F}\times\SL_{2}(\C)$ when $\G=\Sp_{2n}$.

We put
\begin{itemize}
\item
$\Pi_{\temp}(\Sp_{2n})$ to be the set of equivalence classes of irreducible tempered representations of $\Sp_{2n}(F)$, and
\item
$\Phi_{\temp}(\Sp_{2n})$ to be the set of $\hat{\G}$-conjugacy classes of tempered (i.e., the image of $W_{F}$ is bounded in $\hat{\G}$) $L$-parameters of $\Sp_{2n}$.
\end{itemize}
For any $\phi\in\Phi_{\temp}(\Sp_{2n})$, we define a finite group $\mathcal{S}_{\phi}$ to be the group of connected components of the centralizer group of $\mathrm{Im}(\phi)$ in $\SO_{2n+1}(\C)$:
\begin{align*}
\mathcal{S}_{\phi}
&:=\pi_{0}\bigl(\Cent_{\SO_{2n+1}(\C)}(\mathrm{Im}(\phi))\bigr)\\
&=\Cent_{\SO_{2n+1}(\C)}(\mathrm{Im}(\phi))/\Cent_{\SO_{2n+1}(\C)}(\mathrm{Im}(\phi))^{\circ}.
\end{align*}
Note that here we implicitly fix a representative of the $\hat{\G}$-conjugacy class $\phi$ and again write $\phi$ for it by abuse of notation.

The \textit{local Langlands correspondence for tempered representations of $\Sp_{2n}(F)$}, which was established by Arthur (\cite[Theorems 1.5.1 and 2.2.1]{MR3135650}), asserts that there exists a natural map
\[
\LLC_{\Sp_{2n}}\colon\Pi_{\temp}(\Sp_{2n})\rightarrow\Phi_{\temp}(\Sp_{2n}),
\]
which is surjective and with finite fibers.
In other words, by letting $\Pi_{\phi}^{\Sp_{2n}}$ be the fiber at an $L$-parameter $\phi$, we have a natural partition
\[
\Pi_{\temp}(\Sp_{2n})
=
\bigsqcup_{\phi\in\Phi_{\temp}(\Sp_{2n})}\Pi_{\phi}^{\Sp_{2n}}.
\]
For any $\phi\in\Phi_{\temp}(\Sp_{2n})$, the finite set $\Pi_{\phi}^{\Sp_{2n}}$ is called an \textit{$L$-packet} and equipped with a bijective map (with respect to a chosen Whittaker datum of $\Sp_{2n}$) to the set $\mathcal{S}_{\phi}^{\vee}$ of irreducible characters of $\mathcal{S}_{\phi}$.
Each $L$-packet $\Pi_{\phi}^{\Sp_{2n}}$ is characterized via the \textit{endoscopic character relation}, which is explained as follows.
Regarding $\phi$ as a tempered $L$-parameter of $\GL_{2n+1}$, we obtain an irreducible tempered representation $\pi_{\phi}$ of $\GL_{2n+1}(F)$ corresponding to $\phi$ under the local Langlands correspondence for general linear groups.
(The representation $\pi_{\phi}$ is called the \textit{endoscopic lift} of $\Pi_{\phi}^{\Sp_{2n}}$.)
Since $\phi$ is self-dual, so is $\pi_{\phi}$, hence $\pi_{\phi}$ is $\theta$-stable.
Thus we can consider the $\theta$-twisted character $\Theta_{\pi_{\phi},\theta}$ of $\pi_{\phi}$.
Then the $L$-packet $\Pi_{\phi}^{\Sp_{2n}}$ is characterized as the unique finite subset of $\Pi_{\temp}(\Sp_{2n})$ satisfying the identity (``endoscopic character relation'')
\[
\Theta_{\pi_{\phi},\theta}(g)
=
\sum_{\begin{subarray}{c}h\in\Sp_{2n}(F)\\ h\leftrightarrow g \end{subarray}}
\sum_{\pi\in\Pi_{\phi}^{\Sp_{2n}}}\Theta_{\pi}(h)
\]
for any strongly $\theta$-regular $\theta$-semisimple element $g$ of $\GL_{2n+1}(F)$, where the first sum runs over the stable conjugacy classes of strongly regular semisimple elements of $\Sp_{2n}(F)$ which are \textit{norms} of $g$ in the sense of twisted endoscopy (\cite[Section 3.3]{MR1687096}).

\begin{rem}
\begin{enumerate}
%\item
%The local Langlands correspondence can be extended from the tempered case to the non-tempered case utilizing the Langlands classification, which classifies non-tempered representations via tempered representations of Levi subgroups.
%However, we do not explain it here since our interest is in an $L$-packet consisting of simple supercuspidal representations, which is a special case of tempered $L$-packets.
\item
As explained in Section \ref{subsec:theta-char-GL}, the notion of the $\theta$-twisted character depends on the choice of an intertwiner $\pi_{\phi}\cong\pi_{\phi}^{\theta}$.
In the above identity, we implicitly adopt the \textit{Whittaker normalization} of an intertwiner by fixing a $\theta$-stable Whittaker datum of $\GL_{2n+1}$.
On the other hand, recall that we chose an explicit intertwiner for each $\theta$-stable simple supercuspidal representation $\pi^{\GL_{2n+1}}_{a}$ of $\GL_{2n+1}(F)$.
A priori, it is nontrivial whether these two choices of an intertwiner coincide.
However, we can check the coincidence easily; see \cite[Section 5.1]{MR3904769} for the details.
(Note that the explanation in \cite[Section 5.1]{MR3904769} is for $\GL_{2n}$ with odd $p$, but completely the same argument works in our setting.)
\item
In general, the endoscopic character relation also involves a subtle correction term called the \textit{(Langlands--Kottwitz--Shelstad) transfer factor}.
However, it is known that the transfer factor is always trivial in our setting.
Moreover, it can be checked easily that a norm in $\Sp_{2n}(F)$ of a strongly $\theta$-regular $\theta$-semisimple element of $\GL_{2n+1}(F)$ is unique (if exists) up to stable conjugacy.
In other words, the first index set of the endoscopic character relation is a singleton whenever it is not empty.
See \cite[Proposition 7.2]{Oi:2018} for the details.
\end{enumerate}
\end{rem}

\subsection{Descent of supercuspidal representations}\label{subsec:descent}
Now let us consider the $\theta$-stable simple supercuspidal representation $\pi^{\GL_{2n+1}}_{a}$ of $\GL_{2n+1}(F)$ ($a\in k^{\times}$).
Let $\phi_{a}$ be the $L$-parameter of $\GL_{2n+1}$ which corresponds to $\pi^{\GL_{2n+1}}_{a}$ under the local Langlands correspondence for $\GL_{2n+1}$.
As $\pi^{\GL_{2n+1}}_{a}$ is supercuspidal, $\phi_{a}$ is trivial on the $\SL_{2}(\C)$-part and irreducible as a representation of $W_{F}$.
Since $\pi^{\GL_{2n+1}}_{a}$ is self-dual, so is $\phi_{a}$.
Furthermore, as the central character of $\pi^{\GL_{2n+1}}_{a}$ is trivial, the determinant character of $\phi_{a}$ is trivial.
Therefore, by noting that $\phi_{a}$ is irreducible, we may assume that the image of $\phi_{a}$ is contained in $\SO_{2n+1}(\C)$, which is equal to the Langlands dual group of $\Sp_{2n}$.
Accordingly, we may regard $\phi_{a}$ as a tempered $L$-parameter of $\Sp_{2n}$ and get a tempered $L$-packet $\Pi_{\phi_{a}}^{\Sp_{2n}}$.

\begin{prop}\label{prop:supercuspidality}
The $L$-packet $\Pi_{\phi_{a}}^{\Sp_{2n}}$ is a singleton consisting of a supercuspidal representation of $\Sp_{2n}(F)$.
\end{prop}

\begin{proof}
By the irreducibility of $\phi_{a}$, this follows from a result of Xu on a parametrization of supercuspidal representations in an $L$-packet (\cite{MR3713922}).
See \cite[Proposition 5.7]{MR3904769} for the details.
\end{proof}

In the following, let us write $\pi^{\Sp_{2n}}$ for the supercuspidal representation of $\Sp_{2n}(F)$ which belongs to the singleton $L$-packet $\Pi_{\phi_{a}}^{\Sp_{2n}}$.

\subsection{Depth bound of the descended representation}

\begin{prop}\label{prop:char-at-h}
Let $h_{u}\in I^{+}_{\Sp_{2n}}$ be an affine generic element as in Example \ref{example} \eqref{hu}.
Then we have 
\[
\Theta_{\pi^{\Sp_{2n}}}(h_{u})
=\Kl^{n+1}_{au}(\psi).
\]
\end{prop}

\begin{proof}
By the endoscopic character relation $\pi^{\GL_{2n+1}}_{a}$ and $\Pi_{\phi_{a}}^{\Sp_{2n}}$, we have
\[
\Theta_{\pi^{\GL_{2n+1}}_{a},\theta}(g)
=\Theta_{\pi^{\Sp_{2n}}}(h)
\]
for any strongly $\theta$-regular $\theta$-semisimple element $g\in\GL_{2n+1}(F)$ and its norm $h\in\Sp_{2n}(F)$ (see Section \ref{subsec:LLC}).
Let $g_{u}\in\GL_{2n+1}(F)$ be the element as in Example \ref{example} \eqref{gu} and we take $(g,h)$ in this equality to be $(g_{u},h_{u})$.
This is possible since the characteristic polynomial of $g_{u}\theta(g_{u})$ is the product of that of $h_{u}$ and $(T-1)$, where $T$ denotes the variable of the characteristic polynomial, hence $h_{u}$ is a norm of $g_{u}$ in the sense of twisted endoscopy (cf.\ \cite[Section 4.1]{MR3904769}).

Therefore we have
\[
\Theta_{\pi^{\GL_{2n+1}}_{a},\theta}(g_{u})
=\Theta_{\pi^{\Sp_{2n}}}(h_{u}).
\]
If we write $g_{u}\theta(g_{u})=(z_{ij})_{ij}$, then we have $(z_{1,2},\ldots,z_{n,n+1})=(1,\ldots,1)$.
Moreover, the $(2n+1,1)$-entry of $g_{u}$ is given by $\varpi u$.
Hence, by Proposition \ref{prop:twisted character}, we get
\[
\Theta_{\pi^{\GL_{2n+1}}_{a},\theta}(g_{u})
=\Kl^{n+1}_{au}(\psi).
\]
\end{proof}

\begin{prop}\label{prop:char-at-y}
Let $y=(y_{ij})_{ij}\in I^{+}_{\Sp_{2n}}$ be an affine generic element.
Then we have either $\Theta_{\pi^{\Sp_{2n}}}(y)=0$ or 
\[
\Theta_{\pi^{\Sp_{2n}}}(y)
=\Kl^{n+1}_{-ay_{1,2}^{2}\cdots y_{n-1,n}^{2}y_{n,n+1}y_{2n,1}\varpi^{-1}}(\psi).
\]
\end{prop}

\begin{proof}
Recall that any affine generic element $y$ of $I^{+}_{\Sp_{2n}}$ is (strongly) regular semisimple.
Moreover, such a $y$ is necessarily elliptic.
Indeed, in order to check this, it suffices to show that the centralizer group $\Cent_{\Sp_{2n}(F)}(y)$ of $y$ in $\Sp_{2n}(F)$ is compact.
Note that
\[
\Cent_{\Sp_{2n}(F)}(y)
:=\{h\in\Sp_{2n}(F)\mid hyh^{-1}=y\}
\subset \{h\in\Sp_{2n} \mid hyh^{-1}\in I^{+}_{\Sp_{2n}}\}.
\]
By \cite[Lemma 3.8]{Oi:2018}, which is valid even when $p=2$, the right-hand side is given by $I_{\Sp_{2n}}$.
Since $I_{\Sp_{2n}}$ is compact, so is $\Cent_{\Sp_{2n}(F)}(y)$.

In general, it is known that the elliptic (strong) regularity of $y$ implies that there exists a (strongly) $\theta$-regular $\theta$-semisimple element $x$ of $\GL_{2n+1}(F)$ such that $y$ is a norm of $x$.
(This follows from the adjoint relation of the transfer factor; see the proof of \cite[Proposition 2.1.1]{MR3135650}.)
Then the endoscopic character relation implies that
\[
\Theta_{\pi^{\GL_{2n+1}}_{a},\theta}(x)
=\Theta_{\pi^{\Sp_{2n}}}(y).
\]
If $x$ is not $\theta$-conjugate to an element of $ZI_{\GL_{2n+1}}^{+}\lan\varphi_{a^{-1}}^{\GL_{2n+1}}\ran$, then the Frobenius formula for the $\theta$-twisted character (\cite[I.6.2 Th\'eor\`eme]{MR3632513}) implies that $\Theta_{\pi^{\GL_{2n+1}}_{a},\theta}(x)$ is zero.
Thus let us consider the case where $x$ belongs to $ZI_{\GL_{2n+1}}^{+}\lan\varphi_{a^{-1}}^{\GL_{2n+1}}\rangle$.

Note that, for any $z\in Z\cong F^{\times}$, the product $zx$ is also a strongly $\theta$-regular $\theta$-semisimple element of $\GL_{2n+1}(F)$ which has $y$ as its norm.
When $z=\varpi$, we have 
\[
\val\circ\det(zx)=(2n+1)+\val\circ\det(x).
\]
On the other hand, we have 
\[
\val\circ\det(\varphi_{a^{-1}}^{\GL_{2n+1}}x\theta(\varphi_{a^{-1}}^{\GL_{2n+1}})^{-1})=2+\val\circ\det(x).
\]
Therefore, by combining $Z$-translation and $\theta$-conjugacy, we may furthermore assume that $\val\circ\det(x)=0$; this means that $x$ belongs to $Z(q)I_{\GL_{2n+1}}^{+}$.
Here, $Z(q)$ denotes the subgroup of $Z$ consisting of elements of finite prime-to-$p$ order.
Again by translating $x$ via $Z(q)$, we may suppose that $x\in I_{\GL_{2n+1}}^{+}$.

As $y$ is a norm of $x$, the characteristic polynomial of $x\theta(x)$ is given by the product of $(T-1)$ and that of $y$.
Since $y$ is affine generic, its characteristic polynomial is an Eisenstein polynomial in $(T-1)$.
Then $x$ must be $\theta$-affine generic by Proposition \ref{prop:Eisenstein}.
Therefore, if we put $x=(x_{ij})_{ij}$ and $x\theta(x)=(z_{ij})_{ij}$, then Proposition \ref{prop:twisted character} implies that
\[
\Theta_{\pi^{\GL_{2n+1}}_{a},\theta}(x)
=
\Kl^{n+1}_{\alpha}(\psi),
\]
where $\alpha$ is the image of $az_{1,2}^{2}\cdots z_{n,n+1}^{2}x_{2n+1,1}\varpi^{-1}\in\mcO^{\times}$ in the residue field $k$.
By Proposition \ref{prop:Eisenstein}, $z_{1,2}^{2}\cdots z_{n,n+1}^{2}x_{2n+1,1}$ is nothing but the constant term of the Eisenstein polynomial (modulo $\mfp^{2}$).
In terms of $y=(y_{ij})_{ij}$, the constant term is given by $-y_{1,2}\cdots y_{2n-1,2n}y_{2n,1}$ modulo $\mfp^{2}$.
Hence we get
\begin{align*}
\Theta_{\pi^{\GL_{2n+1}}_{a},\theta}(x)
&=
\Kl^{n+1}_{-ay_{1,2}\cdots y_{2n-1,2n}y_{2n,1}\varpi^{-1}}(\psi)\\
&=
\Kl^{n+1}_{-ay_{1,2}^{2}\cdots y_{n-1,n}^{2}y_{n,n+1}y_{2n,1}\varpi^{-1}}(\psi).
\end{align*}
\end{proof}

\begin{cor}\label{cor:depth-bound}
The representation $\pi^{\Sp_{2n}}$ has a nonzero $I_{\Sp_{2n}}^{++}$-fixed vector.
In particular, the representation $\pi^{\Sp_{2n}}$ is either depth-zero or simple supercuspidal.
\end{cor}

\begin{proof}
Let $u\in k^{\times}$ be an element such that $\Kl_{au}^{n+1}(\psi)\neq0$.
The existence of such an element $u\in k^{\times}$ follows from the fact that the Fourier transform of the Kloosterman sums is given by a Gauss sum, which is nonzero; see \cite[Corollary A.5]{MR3904769} for the details.
Then, by Propositions \ref{prop:char-at-h} and \ref{prop:char-at-y},
\begin{itemize}
\item
$\Theta_{\pi^{\Sp_{2n}}}(h_{u})=\Kl_{au}^{n+1}(\psi)\neq0$, and
\item
$\Theta_{\pi^{\Sp_{2n}}}(y)$ is equal to either $0$ or $\Kl_{au}^{n+1}(\psi)$ for any $y\in h_{u}I_{\Sp_{2n}}^{++}$.
\end{itemize}
This implies that $\Theta_{\pi^{\Sp_{2n}}}(\mathbbm{1}_{h_{u}I_{\Sp_{2n}}^{++}})\neq0$, hence we get the first assertion.

As $I_{\Sp_{2n}}^{++}$ is the $(\frac{1}{2n}+)$-th Moy--Prasad filtration of the Iwahori subgroup associated with the barycenter of the fundamental alcove, we conclude that the depth of $\pi^{\Sp_{2n}}$ is not greater than $\frac{1}{2n}$.
Since $\frac{1}{2n}$ is the minimal positive depth of representations of $\Sp_{2n}(F)$ which can be attained only by simple supercuspidal representations, we get the second assertion (see \cite[Appendix B]{Oi:2018} for the details of the discussion here).
\end{proof}

\subsection{A consequence of the formal degree conjecture}\label{subsec:FDC}

We say that a tempered $L$-parameter $\phi\in\Phi_{\temp}(\Sp_{2n})$ is \textit{discrete} if its centralizer group $\Cent_{\SO_{2n+1}(\C)}(\phi)$ is finite.
It is known that $\phi$ is discrete if and only if $\Pi_{\phi}^{\Sp_{2n}}$ contains a discrete series representation of $\Sp_{2n}(F)$, and that, in this case, every member of $\Pi_{\phi}^{\Sp_{2n}}$ is discrete series.
Note that hence the $L$-parameter $\phi_{a}$ of our interest is discrete in this sense.

For discrete $L$-parameters, Hiraga--Ichino--Ikeda proposed the following conjecture (\cite[Conjecture 1.4]{MR2350057}): (here we state the conjecture according to a formulation by Gross--Reeder, \cite[Conjecture 7.1 (5)]{MR2730575}):

\begin{conj}[Formal degree conjecture]\label{conj:FDC}
Let $\phi\in\Phi_{\temp}(\Sp_{2n})$ be a discrete $L$-parameter.
Then, for any $\pi\in\Pi_{\phi}^{\Sp_{2n}}$, we have
\[
|\deg(\pi)|
=
\frac{1}{|\mathcal{S}_{\phi}|}\cdot\frac{|\gamma(0,\Ad\circ\phi,\psi_{F})|}{|\gamma(0,\Ad\circ\phi_{0},\psi_{F})|}.
\]
Here, 
\begin{itemize}
\item
$\deg(\pi)$ is the formal degree of $\pi$ with respect to the Euler--Poincare measure (see \cite[Section 7.1]{MR2730575}),
\item
$\Ad$ is the adjoint representation of $\SO_{2n+1}(\C)$ on its Lie algebra $\mathfrak{so}_{2n+1}(\C)$,
\item
$\gamma(s,-,\psi_{F})$ is the $\gamma$-factor for representations of $W_{F}$ with respect to a nontrivial additive character $\psi_{F}$ of $F$ of level $0$, and
\item
$\phi_{0}$ denotes the principal parameter in the sense of Gross--Reeder (see \cite[Section 3.3]{MR2730575}).
\end{itemize}
\end{conj}

\begin{rem}
In \cite{MR2350057}, the formal degree conjecture is formulated for any quasi-split connected reductive group $\G$.
In general, the right-hand side of the identity of Conjecture \ref{conj:FDC} must contain one more term ``$\langle1,\pi\rangle$'' (see \cite[Conjecture 1.4]{MR2350057}).
Here $\langle-,\pi\rangle$ denotes the irreducible character of $\mathcal{S}_{\phi}$ corresponding to $\pi$ (recall that each $L$-packet is equipped with a bijective map to the set of irreducible characters of the finite group $\mathcal{S}_{\phi}$).
In fact, the group $\mathcal{S}_{\phi}$ is always abelian when $\G=\Sp_{2n}$.
Accordingly, $\langle1,\pi\rangle$ is always given by $1$.
\end{rem}

The formal degree conjecture is still open in general.
However, recently Beuzart-Plessis announced that he proved it for $\Sp_{2n}$ (in his talk at the seminar ``S\'eminaire Groupes R\'eductifs et Formes Automorphes'', held on November 8, 2021; \cite{Beuzart-Plessis:2022}).
%\textcolor{blue}{(MO: In general, how to give a reference to this kind of statement?)}
In the following, we investigate what can be proved by assuming this conjecture.

We start with reviewing a description of the $L$-parameter $\phi_{a}$ of the simple supercuspidal representation $\pi^{\GL_{2n+1}}_{a}$ according to Bushnell--Henniart:

\begin{prop}\label{prop:L-par}
As a $(2n+1)$-dimensional representation of $W_{F}$, we have
\[
\phi_{a}\cong \Ind_{W_{K}}^{W_{F}} \xi,
\]
where 
\begin{itemize}
\item
$K$ is a totally ramified extension of $F$ of degree $2n+1$, and
\item
$\xi\colon W_{K}\rightarrow\C^{\times}$ is a quadratic character of Swan conductor $1$.
\end{itemize}
\end{prop}

\begin{proof}
By \cite{MR3158004}, we have $\phi_{a}\cong \Ind_{W_{K}}^{W_{F}} \xi$ for a totally ramified extension $K$ of $F$ of degree $2n+1$ and a character $\xi\colon W_{K}\rightarrow\C^{\times}$ of Swan conductor $1$.
Since $\phi_{a}$ is self-dual, $\xi$ is necessarily quadratic by \cite[Lemma 3.2]{MR2824846}.
\end{proof}

\begin{rem}\label{rem:explicit}
Although Bushnell--Henniart give a complete description of the character $\xi$ in \cite{MR3158004}, we do not review it here since we will only need the fact that $\xi$ is quadratic.
\end{rem}

Let us compute the quantity $|\gamma(0,\Ad\circ\phi,\psi_{F})|$ based on this description of $\phi_{a}$.
Note that we have
\[
\Ad\circ\phi
\cong 
\wedge^{2}\phi,
\]
where $\phi$ is viewed as a homomorphism $W_{F}\rightarrow\SO_{2n+1}(\C)$ on the left-hand side and as a $(2n+1)$-dimensional representation on the right-hand side.

\begin{lem}\label{lem:L}
We have $L(s,\Ad\circ\phi_{a})=1$.
\end{lem}

To prove this lemma, let us first show the following lemma, which might be well-known to experts:

\begin{lem}\label{lem:Clifford}
For any irreducible representation $\phi$ of $W_{F}$, the following two numbers coincide:
\begin{enumerate}
\item
the number of irreducible constituents of the restriction $\phi|_{I_{F}}$ of $\phi$ to the inertia subgroup $I_{F}$;
\item
the maximal degree of an unramified extension $E$ of $F$ such that there exists an irreducible representation $\sigma$ of $W_{E}$ satisfying $\phi\cong\Ind_{W_{E}}^{W_{F}}\sigma$.
\end{enumerate}
\end{lem}

\begin{proof}
Let $d$ be the number of irreducible constituents of the restriction $\phi|_{I_{F}}$ and $\sigma$ an irreducible constituent of $\phi|_{I_{F}}$.
We let $W_{E}$ denote the stabilizer of $\sigma$ in $W_{F}$, i.e.,
\[
W_{E}:=\{w\in W_{F} \mid \text{$\sigma^{w}\cong\sigma$ as a representation of $W_{E}$}\}.
\]
Then, by Clifford theory, there exists an irreducible constituent $\tau$ of $\phi|_{W_{E}}$ such that $\tau|_{I_{F}}$ is $\sigma$-isotypic and $\phi\cong\Ind_{W_{E}}^{W_{F}}\tau$ (thus $E$ is a finite unramified extension of $F$).
By Mackey theory, we see that
\[
\phi|_{I_{F}}
\cong(\Ind_{W_{E}}^{W_{F}}\tau)|_{I_{F}}
\cong \bigoplus_{w\in W_{F}/W_{E}}\tau^{w}|_{I_{F}},
\]
where $\tau^{w}$ is the representation of $W_{E}$ given by $\tau^{w}(w'):=\tau(w^{-1}w'w)$ for $w'\in W_{E}$.

Note that $\sigma$ extends to a representation $\tilde{\sigma}$ of its stabilizer group $W_{E}$ since $W_{E}/I_{F}$ is cyclic; 
for example, if we fix an intertwiner $I_{w}\colon \sigma^{w}\cong\sigma$ as a representation of $W_{E}$ for a generator $w$ of $W_{E}/I_{F}$, then $\tilde{\sigma}(w^{k}w'):=I_{w}^{k}\circ\sigma(w')$ (for $k\in\Z$ and $w'\in I_{F}$) gives an extension of $\sigma$ to $W_{E}$.
Since $\Hom_{I_{F}}(\tau,\sigma)\neq0$, Frobenius reciprocity implies that $\Hom_{W_{E}}(\tau,\Ind_{I_{F}}^{W_{E}}\sigma)\neq0$.
By the projection formula, we have
\[
\Ind_{I_{F}}^{W_{E}}\sigma
=\Ind_{I_{F}}^{W_{E}}(\tilde{\sigma}|_{I_{F}})
\cong\tilde{\sigma}\otimes\Ind_{I_{F}}^{W_{E}}\mathbbm{1}
\cong\tilde{\sigma}\otimes\biggl(\bigoplus_{\chi\in (W_{E}/I_{F})^{\vee}}\chi\biggr).
\]
(Note that $W_{E}/I_{F}$ is cyclic.)
Hence, by the irreducibility of $\tau$, $\tau$ is isomorphic to $\tilde{\sigma}\otimes\chi$ for some character $\chi$ of $W_{E}/I_{F}$.
In particular, $\tau^{w}|_{I_{F}}$ is irreducible for any $w\in W_{F}/W_{E}$.
In other words, $d$ equals the degree of the extension $E/F$.

Conversely, if $\phi$ is induced from a representation $\rho$ of $W_{K}$, where $K$ is an unramified extension of $F$ of degree $e$, then we have 
\[
\phi|_{I_{F}}
\cong (\Ind_{W_{K}}^{W_{F}}\rho)|_{I_{F}}
\cong \bigoplus_{w\in W_{F}/W_{K}}\rho^{w}|_{I_{F}}
\]
by Mackey theory.
Hence $e$ divides $d$.
\end{proof}

\begin{proof}[Proof of Lemma \ref{lem:L}]
By definition, we have
\[
L(s,\Ad\circ\phi_{a})
=
\det\bigl(1-\wedge^{2}\phi_{a}(\Frob) \mid (\wedge^{2}\phi_{a})^{I_{F}}\bigr)^{-1}.
\]
Hence it suffices to show that $(\wedge^{2}\phi_{a})^{I_{F}}=0$.

According to the description of $\phi_{a}$ as in Proposition \ref{prop:L-par}, we see that the number as in (2) of Lemma \ref{lem:Clifford} is equal to $1$.
Hence, Lemma \ref{lem:Clifford} implies that $\phi_{a}|_{I_{F}}$ is irreducible.

Thus, by Schur's lemma, the space $\Hom_{I_{F}}(\phi_{a},\phi_{a})\cong(\phi_{a}\otimes\phi_{a})^{I_{F}}$ is $1$-dimensional (note that $\phi_{a}\cong\phi_{a}^{\vee}$).
Since we have 
\[
(\phi_{a}\otimes\phi_{a})^{I_{F}}
\cong
(\Sym^{2}\phi_{a})^{I_{F}}\oplus(\wedge^{2}\phi_{a})^{I_{F}}
\]
and $(\Sym^{2}\phi_{a})^{I_{F}}$ is $1$-dimensional by the orthogonality of $\phi_{a}$, we conclude that $(\wedge^{2}\phi_{a})^{I_{F}}=0$.
%\textcolor{blue}{(MO: Guy taught me that this lemma holds in a more general setting on 22th June. I should record his proof somewhere.)}
\end{proof}

\begin{prop}\label{prop:Swan}
We have $\Swan(\Ad\circ\phi_{a})=n$.
\end{prop}

\begin{proof}
If we can show the following two equalities, then we get the desired equality:
\begin{enumerate}
\item
$\Swan(\Sym^{2}\phi_{a})+\Swan(\wedge^{2}\phi_{a})=2n$,
\item
$\Swan(\Sym^{2}\phi_{a})-\Swan(\wedge^{2}\phi_{a})=0$.
\end{enumerate}

By noting that $\Swan(\Sym^{2}\phi_{a})+\Swan(\wedge^{2}\phi_{a})=\Swan(\phi_{a}\otimes\phi_{a}^{\vee})$, the first equality follows from an explicit formula for the conductor of the Rankin--Selberg convolution due to Bushnell--Henniart--Kutzko (\cite{MR1606410}) as follows.
Let $[\mathfrak{a},1,0,\beta]$ be a simple stratum associated with the simple supercuspidal representation $\pi^{\GL_{2n+1}}_{a}$ of $\GL_{2n+1}(F)$ (see \cite[434 page]{MR3158004}).
According to \cite[6.5 Theorem (i)]{MR1606410}, we get
\[
\Artin(\phi_{a}\otimes\phi_{a}^{\vee})
=
(2n+1)^{2}\biggl(1+\frac{\mathfrak{c}(\beta)}{(2n+1)^{2}}\biggr)-1,
\]
where $\Artin(-)$ denotes the Artin conductor.
Here $\mathfrak{c}(\beta)$ is the quantity introduced in \cite[6.4]{MR1606410}.
As $\beta$ is minimal, it is not hard to see that $\mathfrak{c}(\beta)=2n$ (cf.\ \cite[6.12]{MR1606410}).
Hence $\Artin(\phi_{a}\otimes\phi_{a}^{\vee})=(2n+1)^{2}+2n-1$.
By noting that $\dim_{\C}(\phi_{a}\otimes\phi_{a}^{\vee})=(2n+1)^{2}$ and $\dim_{\C}((\phi_{a}\otimes\phi_{a}^{\vee})^{I_{F}})=1$ (see the proof of Lemma \ref{lem:L}), we have 
\begin{align*}
\Artin(\phi_{a}\otimes\phi_{a}^{\vee})
&=
\dim_{\C}\bigl((\phi_{a}\otimes\phi_{a}^{\vee})/(\phi_{a}\otimes\phi_{a}^{\vee})^{I_{F}}\bigr)+\Swan(\phi_{a}\otimes\phi_{a}^{\vee})\\
&=
(2n+1)^{2}-1+\Swan(\phi_{a}\otimes\phi_{a}^{\vee}).
\end{align*}
Thus, by comparing the two equalities, we get the desired equality $\Swan(\Sym^{2}\phi_{a})+\Swan(\wedge^{2}\phi_{a})=2n$.

%\textcolor{blue}{MO: Probably this explanation should be modified. (My understanding of the mentioned results is quite shallow!)}

Let us check the latter equality (2).
By noting that the Swan conductor depends only on the wild ramification, we consider the restriction to the wild inertia subgroup $P_{F}$.
Since we have $\phi_{a}\cong\Ind_{W_{K}}^{W_{F}}\xi$ as in Proposition \ref{prop:L-par}, we get
\[
\phi_{a}|_{P_{F}}\cong
\bigoplus_{w\in W_{F}/W_{K}} {}^{w}\xi|_{P_{F}}
\]
by Mackey theory (note that $K/F$ is tamely ramified).
Hence we have
\[
\Sym^{2}\phi_{a}|_{P_{F}}
\cong
\bigoplus_{w,w'\in W_{F}/W_{K}} ({}^{w}\xi|_{P_{F}})\cdot({}^{w'}\xi|_{P_{F}})
\]
and
\[
\wedge^{2}\phi_{a}|_{P_{F}}
\cong
\bigoplus_{\begin{subarray}{c}w,w'\in W_{F}/W_{K}\\ w\neq w'\end{subarray}} ({}^{w}\xi|_{P_{F}})\cdot({}^{w'}\xi|_{P_{F}}).
\]
This implies that we have
\[
(\Sym^{2}\phi_{a}-\wedge^{2}\phi_{a})|_{P_{F}}
\cong
\bigoplus_{w\in W_{F}/W_{K}} ({}^{w}\xi|_{P_{F}})^{2}.
\]
However, since $\xi$ is quadratic, every summand is trivial.
Thus we get $\Swan(\Sym^{2}\phi_{a})-\Swan(\wedge^{2}\phi_{a})=0$.
\end{proof}

\begin{prop}\label{prop:fdeg}
We have $|\gamma(0,\Ad\circ\phi,\psi_{F})|=q^{n^{2}+n}$.
In particular, $|\deg(\pi^{\Sp_{2n}})|=|\deg(\pi^{\Sp_{2n}}_{b})|$ for any $b\in k^{\times}$.
\end{prop}

\begin{proof}
Recall that
\[
\gamma(0,\Ad\circ\phi_{a},\psi_{F})
=\varepsilon(0,\Ad\circ\phi_{a},\psi_{F})\cdot\frac{L(1,\Ad\circ\phi_{a})}{L(0,\Ad\circ\phi_{a})}
\]
by definition.
As $\psi_{F}$ is taken to be of level $0$, we have
\[
|\varepsilon(0,\Ad\circ\phi_{a},\psi_{F})|
=q^{\frac{1}{2}\Artin(\Ad\circ\phi_{a})}
\]
(see \cite[the equality (10) and Proposition 2.3]{MR2730575}).
By noting that $\dim_{\C}(\Ad\circ\phi_{a})=n(2n+1)$ and $\dim_{\C}((\Ad\circ\phi_{a})^{I_{F}})=0$ (see the proof of Lemma \ref{lem:L}), we have 
\begin{align*}
\Artin(\Ad\circ\phi_{a})
&=\dim_{\C}\bigl((\Ad\circ\phi_{a})/(\Ad\circ\phi_{a})^{I_{F}}\bigr)+\Swan(\Ad\circ\phi_{a})\\
&=n(2n+1)+n
=2(n^{2}+n),
\end{align*}
where we used Proposition \ref{prop:Swan} in the second equality.
Hence we get $|\varepsilon(0,\Ad\circ\phi_{a},\psi_{F})|=q^{n^{2}+n}$.
On the other hand, the contribution of the $L$-factor is trivial by Lemma \ref{lem:L}.
Thus we get the first assertion.

Since $|\mathcal{S}_{\phi_{a}}|=1$, the formal degree conjecture for $\Sp_{2n}$ implies that
\begin{align*}
|\deg(\pi^{\Sp_{2n}})|
&=\frac{|\gamma(0,\Ad\circ\phi_{a},\psi_{F})|}{|\gamma(0,\Ad\circ\phi_{0},\psi_{F})|}\\
&=q^{n^{2}+n}\cdot |\gamma(0,\Ad\circ\phi_{0},\psi_{F})|^{-1}.
\end{align*}
On the other hand, as computed in \cite[(72)]{MR2730575}, the absolute value of the formal degree of a(ny) simple supercuspidal representation $\pi^{\Sp_{2m}}_{b}$ ($b\in k^{\times}$) is given by
\[
|\deg(\pi^{\Sp_{2n}}_{b})|
=
\frac{q^{N+\ell}}{|Z_{\Sp_{2n}}(q)|\cdot |\gamma(0,\Ad\circ\phi_{0},\psi_{F})|},
\]
where 
\begin{itemize}
\item
$N$ is the number of positive roots in $\Sp_{2n}$, hence $n^{2}$,
\item
$\ell$ is the rank of $\Sp_{2n}$, hence $n$, and
\item
$|Z_{\Sp_{2n}}(q)|$ is the number of central elements of $\Sp_{2n}(F)$ of finite prime-to-$p$ order, hence $1$.
\end{itemize}
Therefore we get $|\deg(\pi^{\Sp_{2n}})|=|\deg(\pi^{\Sp_{2n}}_{b})|$.
\end{proof}

\begin{cor}\label{cor:ssc}
The representation $\pi^{\Sp_{2n}}$ is simple supercuspidal.
\end{cor}

\begin{proof}
By Corollary \ref{cor:depth-bound}, $\pi^{\Sp_{2n}}$ is simple supercuspidal or depth-zero supercuspidal.
As observed in \cite[A.4]{Henniart:2022}, the formal degree of a simple supercuspidal representation of $\Sp_{2n}(F)$ cannot be equal to that of any depth-zero supercuspidal representation of $\Sp_{2n}(F)$.
Thus the equality $|\deg(\pi^{\Sp_{2n}})|=|\deg(\pi^{\Sp_{2n}}_{b})|$ of Proposition \ref{prop:fdeg} (for any $b\in k^{\times}$) implies that $\pi^{\Sp_{2n}}$ is necessarily simple supercuspidal.
\end{proof}

\subsection{Endoscopic lifting of simple supercuspidal representations}
By Corollary \ref{cor:ssc}, the descended representation $\pi^{\Sp_{2n}}$ can be written as $\pi^{\Sp_{2n}}_{b}$ for some $b\in k^{\times}$.

\begin{prop}\label{prop:exact-corresp}
We have $b=a$.
\end{prop}

\begin{proof}
By Proposition \ref{prop:char-at-h}, we have
\[
\Theta_{\pi^{\Sp_{2n}}_{b}}(h_{u})
=\Kl_{au}^{n+1}(\psi)
\]
for any $u\in k^{\times}$.
Since the left-hand side is given by $\Kl_{bu}^{n+1}(\psi)$ by Proposition \ref{prop:char-Sp} (note that $-1=1$ in $k$), we get the equality $\Kl_{au}^{n+1}(\psi)=\Kl_{bu}^{n+1}(\psi)$, which holds for any $u\in k^{\times}$.
Then we can conclude that $a=b$ (see \cite[Proposition A.6]{MR3904769}).
\end{proof}

Let us summarize our results:
\begin{thm}\label{thm:main}
Let $F$ be a dyadic field.
For $a\in k^{\times}$, let $\pi^{\Sp_{2n}}_{a}$ be the simple supercuspidal representation as in Section \ref{subsec:ssc-Sp}.
\begin{enumerate}
\item
The $L$-packet of $\Sp_{2n}$ containing $\pi^{\Sp_{2n}}_{a}$ is a singleton.
\item
The endoscopic lift of the $L$-packet $\{\pi^{\Sp_{2n}}_{a}\}$ to $\GL_{2n+1}$ is given by the $\theta$-stable simple supercuspidal representation $\pi^{\GL_{2n+1}}_{a}$ with trivial central character.
\end{enumerate}
\end{thm}

As we already discussed in Section \ref{subsec:descent}, the $L$-parameter of $\pi_{a}^{\Sp_{2n}}$ is trivial on the $\SL_{2}(\C)$-part and irreducible self-dual orthogonal as a $(2n+1)$-dimensional representation of $W_{F}$.
Let us record this observation here:

\begin{cor}\label{cor:main}
Let $F$ be a dyadic field.
Then the $L$-parameter of any simple supercuspidal representation of $\Sp_{2n}(F)$ is irreducible as a $(2n+1)$-dimensional representation of $W_{F}$.
\end{cor}


\begin{thebibliography}{ABPS16}

\bibitem[ABPS16]{MR3618046}
A.-M. Aubert, P.~Baum, R.~Plymen, and M.~Solleveld, \emph{Depth and the local
  {L}anglands correspondence}, Arbeitstagung {B}onn 2013, Progr. Math., vol.
  319, Birkh\"auser/Springer, Cham, 2016, pp.~17--41.

\bibitem[AK19]{MR4031300}
M.~Adrian and E.~Kaplan, \emph{The {L}anglands parameter of a simple
  supercuspidal representation: symplectic groups}, Ramanujan J. \textbf{50}
  (2019), no.~3, 589--619.

\bibitem[Art13]{MR3135650}
J.~Arthur, \emph{The endoscopic classification of representations: Orthogonal
  and symplectic groups}, American Mathematical Society Colloquium
  Publications, vol.~61, American Mathematical Society, Providence, RI, 2013.

\bibitem[BH05a]{MR2138141}
C.~J. Bushnell and G.~Henniart, \emph{The essentially tame local {L}anglands
  correspondence. {I}}, J. Amer. Math. Soc. \textbf{18} (2005), no.~3,
  685--710.

\bibitem[BH05b]{MR2148193}
\bysame, \emph{The essentially tame local {L}anglands correspondence. {II}.
  {T}otally ramified representations}, Compos. Math. \textbf{141} (2005),
  no.~4, 979--1011.

\bibitem[BH10]{MR2679700}
\bysame, \emph{The essentially tame local {L}anglands correspondence, {III}:
  the general case}, Proc. Lond. Math. Soc. (3) \textbf{101} (2010), no.~2,
  497--553.

\bibitem[BH11]{MR2824846}
\bysame, \emph{Self-dual representations of some dyadic groups}, Math. Ann.
  \textbf{351} (2011), no.~1, 67--80.

\bibitem[BH14]{MR3158004}
\bysame, \emph{Langlands parameters for epipelagic representations of {${\rm
  GL}_n$}}, Math. Ann. \textbf{358} (2014), no.~1-2, 433--463.

\bibitem[BHK98]{MR1606410}
C.~J. Bushnell, G.~M. Henniart, and P.~C. Kutzko, \emph{Local
  {R}ankin-{S}elberg convolutions for {${\rm GL}_n$}: explicit conductor
  formula}, J. Amer. Math. Soc. \textbf{11} (1998), no.~3, 703--730.

\bibitem[BP21]{Beuzart-Plessis:2022}
R.~Beuzart-Plessis, \emph{Sur la conjecture du degr{\'e} formel pour les
  groupes classiques}, talk in ``S{\'e}minaire Groupes R{\'e}ductifs et Formes
  Automorphes'', \url{https://www.imj-prg.fr/gestion/evenement/affSeance/8333},
  November 8, 2021.

\bibitem[GGP12]{MR3202556}
W.~T. Gan, B.~H. Gross, and D.~Prasad, \emph{Symplectic local root numbers,
  central critical {$L$} values, and restriction problems in the representation
  theory of classical groups}, Ast\'erisque (2012), no.~346, 1--109, Sur les
  conjectures de Gross et Prasad. I.

\bibitem[GR10]{MR2730575}
B.~H. Gross and M.~Reeder, \emph{Arithmetic invariants of discrete {L}anglands
  parameters}, Duke Math. J. \textbf{154} (2010), no.~3, 431--508.

\bibitem[HC70]{MR0414797}
Harish-Chandra, \emph{Harmonic analysis on reductive {$p$}-adic groups},
  Lecture Notes in Mathematics, Vol. 162, Springer-Verlag, Berlin-New York,
  1970, Notes by G. van Dijk.

\bibitem[Hen00]{MR1738446}
G.~Henniart, \emph{Une preuve simple des conjectures de {L}anglands pour {${\rm
  GL}(n)$} sur un corps {$p$}-adique}, Invent. Math. \textbf{139} (2000),
  no.~2, 439--455.

\bibitem[Hen22]{Henniart:2022}
\bysame, \emph{{C}uspidal representations of {${\rm Sp}(2n,\mathbb{Q}_2)$} with
  irreducible {G}alois parameter}, preprint, appendix to a paper by
  M{\'\i}nguez and S{\'e}cherre in preparation, 2022.

\bibitem[HII08]{MR2350057}
K.~Hiraga, A.~Ichino, and T.~Ikeda, \emph{Formal degrees and adjoint
  {$\gamma$}-factors}, J. Amer. Math. Soc. \textbf{21} (2008), no.~1, 283--304.

\bibitem[HT01]{MR1876802}
M.~Harris and R.~Taylor, \emph{The geometry and cohomology of some simple
  {S}himura varieties}, Annals of Mathematics Studies, vol. 151, Princeton
  University Press, Princeton, NJ, 2001, With an appendix by Vladimir G.
  Berkovich.

\bibitem[KS99]{MR1687096}
R.~E. Kottwitz and D.~Shelstad, \emph{Foundations of twisted endoscopy},
  Ast\'erisque (1999), no.~255, vi+190.

\bibitem[LH17]{MR3632513}
B.~Lemaire and G.~Henniart, \emph{Repr\'esentations des espaces tordus sur un
  groupe r\'eductif connexe {$\mathfrak{p}$}-adique}, Ast\'erisque (2017),
  no.~386, ix+366.

\bibitem[Oi18]{Oi:2018}
M.~Oi, \emph{Simple supercuspidal {$L$}-packets of quasi-split classical
  groups}, \url{arXiv:1805.01400}, to appear in Mem. Amer. Math. Soc., 2018.

\bibitem[Oi19]{MR3904769}
\bysame, \emph{Endoscopic lifting of simple supercuspidal representations of
  {${\rm SO}_{2n+1}$} to {${\rm GL}_{2n}$}}, Amer. J. Math. \textbf{141}
  (2019), no.~1, 169--217.

\bibitem[Pra99]{MR1670568}
D.~Prasad, \emph{Some remarks on representations of a division algebra and of
  the {G}alois group of a local field}, J. Number Theory \textbf{74} (1999),
  no.~1, 73--97.

\bibitem[RY14]{MR3164986}
M.~Reeder and J.-K. Yu, \emph{Epipelagic representations and invariant theory},
  J. Amer. Math. Soc. \textbf{27} (2014), no.~2, 437--477.

\bibitem[Sal88]{MR1039842}
P.~J. Sally, Jr., \emph{Some remarks on discrete series characters for
  reductive {$p$}-adic groups}, Representations of {L}ie groups, {K}yoto,
  {H}iroshima, 1986, Adv. Stud. Pure Math., vol.~14, Academic Press, Boston,
  MA, 1988, pp.~337--348.

\bibitem[Xu17]{MR3713922}
B.~Xu, \emph{On the cuspidal support of discrete series for {$p$}-adic
  quasisplit {$Sp(N)$} and {$SO(N)$}}, Manuscripta Math. \textbf{154} (2017),
  no.~3-4, 441--502.

\end{thebibliography}
\end{document}